\documentclass[11pt,reqno]{amsart}
\usepackage[a4paper,margin=27mm]{geometry}
\usepackage[T1]{fontenc}
\usepackage{lmodern,amssymb,mathtools,booktabs,longtable,microtype}
\usepackage{listings,xcolor,tabularx,array}
\relax {}
\usepackage{mathtools}
\usepackage{xcolor}
\newcolumntype{Y}{>{\raggedright\arraybackslash}X}

\usepackage[colorlinks=true,linkcolor=blue,citecolor=blue,urlcolor=blue]{hyperref}
\hypersetup{pdftitle={Sharp anticanonical volume bounds for Gorenstein Fano fourfolds of Picard number one},
  pdfauthor={Pinxian Bie and Zhengjie Yu},pdfsubject={Revised research manuscript},pdfkeywords={Fano fourfolds, anticanonical volume, Gorenstein singularities}}
\allowdisplaybreaks[2]
\newtheorem{theorem}{Theorem}[section]
\newtheorem{proposition}[theorem]{Proposition}
\newtheorem{lemma}[theorem]{Lemma}
\newtheorem{corollary}[theorem]{Corollary}
\theoremstyle{definition}
\newtheorem{definition}[theorem]{Definition}
\newtheorem{example}[theorem]{Example}
\theoremstyle{remark}
\newtheorem{remark}[theorem]{Remark}
\DeclareMathOperator{\Cl}{Cl}
\DeclareMathOperator{\Pic}{Pic}

\DeclareMathOperator{\rk}{rk}
\DeclareMathOperator{\Supp}{Supp}
\DeclareMathOperator{\Sym}{Sym}
\DeclareMathOperator{\End}{End}
\DeclareMathOperator{\Proj}{Proj}
\DeclareMathOperator{\Spec}{Spec}
\DeclareMathOperator{\im}{im}
\DeclareMathOperator{\lcm}{lcm}
\DeclareMathOperator{\codim}{codim}
\DeclareMathOperator{\NCart}{NCart}
\newcommand{\Disc}{\widehat{\operatorname{Disc}}}
\newcommand{\QQ}{\mathbb Q}
\newcommand{\ZZ}{\mathbb Z}
\newcommand{\CC}{\mathbb C}
\newcommand{\PP}{\mathbb P}
\newcommand{\OO}{\mathcal O}
\newcommand{\TT}{\mathcal T}
\newcommand{\EE}{\mathcal E}
\newcommand{\FF}{\mathcal F}
\newcommand{\Chat}{\widehat C}
\newcommand{\qsim}{\sim_{\QQ}}
\newcommand{\J}{\mathcal J}
\newcommand{\reg}{\mathrm{reg}}

\title[Volumes of Gorenstein Fano fourfolds]{Sharp anticanonical volume bounds for Gorenstein Fano fourfolds of Picard number one}
\author{Pinxian Bie}
\address{\normalfont School of Mathematical Sciences, Fudan University, Shanghai, China}
\email{23110180002@m.fudan.edu.cn}
\author{Zhengjie Yu}
\address{\normalfont School of Mathematical Sciences, Fudan University, Shanghai, China}
\email{23110180053@m.fudan.edu.cn}
\date{October 25, 2026}
\subjclass[2020]{Primary 14J45; Secondary 14E30, 14F06, 32M25}
\keywords{Fano fourfold, anticanonical volume, Gorenstein singularities}

\begin{document}
\begin{abstract}
Let \(X\) be a complex \(\QQ\)-factorial Gorenstein Fano fourfold with \(\rho(X)=1\). If \(X\) has canonical singularities, we prove that \((-K_X)^4\le3528\), with equality if and only if \(X\simeq\PP(1,1,12,28,42)\). In fact, \((-K_X)^4\ge3130\) already implies this isomorphism. If \(X\) has terminal singularities, we prove the sharp bound \((-K_X)^4\le648\), with equality if and only if \(X\simeq\PP(1,1,1,1,2)\). In the canonical case, we also obtain the sharp bound \((-K_X)^4\le1024\) when a positive generator of \(\Cl(X)/\mathrm{tors}\) is Cartier outside a finite set. The proofs use jet estimates involving all Harder--Narasimhan slopes of the tangent sheaf, the geometry of algebraic foliations, and the integrality of intersection numbers on a general leaf. The equality case in the canonical bound is determined by two computations of divisorial section rings.
\end{abstract}
\maketitle
\setcounter{tocdepth}{1}
\tableofcontents
\clearpage

\section{Introduction}\label{sec:intro}

Throughout this paper, we work over the field of complex numbers $\CC$.

The birational geometry of Fano varieties has been intensively studied in recent years. For fixed \(n\) and \(\varepsilon>0\), Birkar's boundedness theorem gives a uniform upper bound for the anticanonical volumes of \(\varepsilon\)-log canonical Fano \(n\)-folds \cite[Theorem 1.1]{BAB}. A natural question is to determine an explicit upper bound for the anticanonical volume, and in particular to find the optimal bound and characterize its equality cases. Addressing this question requires more explicit geometric and numerical invariants than boundedness alone.

Several sharp bounds have been established in low dimensions. In dimension three, Prokhorov proved the sharp degree bound
\(72\) in the canonical Gorenstein case \cite{ProkhorovDegree}. Recently, Jiang--Zhang--Zou showed that any canonical weak Fano
threefold has anticanonical volume at most \(72\)
\cite[Theorem 1.1]{JZZ}, while Jiang--Liu established the upper bound \(66\) for its
Weil Fano index \cite[Theorem 1.1]{JL}. Their work and that of
Jiang--Liu--Liu \cite{JLL} combine singular Riemann--Roch formulas,
Harder--Narasimhan filtrations, and the geometry of algebraic
foliations. In dimension four, Birkar obtained the first explicit general upper bound for canonical Fano varieties in 
\cite[Theorem 1.1]{Birkar4} without assuming the Gorenstein condition,
\(\QQ\)-factoriality, or Picard number one. That bound is extremely large and is not expected to be close to optimal.

In this paper, we study \(\QQ\)-factorial Gorenstein Fano fourfolds of Picard number one. In the canonical case, we establish the sharp upper bound \(3528\) for the anticanonical volume and show that equality holds uniquely for the weighted projective space \(\PP(1,1,12,28,42)\) (Theorem~\ref{thm:canonical}); in fact, any such fourfold with \((-K_X)^4\ge 3130\) is already isomorphic to \(\PP(1,1,12,28,42)\). In the terminal case, we prove the sharp bound \(648\), attained uniquely by \(\PP(1,1,1,1,2)\) (Theorem~\ref{thm:terminal}). Our first result states the terminal bound.

\begin{theorem}\label{thm:terminal}
Let \(X\) be a \(\QQ\)-factorial Gorenstein terminal Fano fourfold with \(\rho(X)=1\). Then
\[
                         (-K_X)^4\le648.
\]
Moreover, equality holds if and only if \(X\simeq\PP(1,1,1,1,2)\).
\end{theorem}

Our second main result gives the optimal bound in the canonical case.

\begin{theorem}\label{thm:canonical}
Let \(X\) be a \(\QQ\)-factorial Gorenstein canonical Fano fourfold with \(\rho(X)=1\). Then
\[
                         (-K_X)^4\le3528.
\]
Moreover, equality holds if and only if \(X\simeq\PP(1,1,12,28,42)\).
\end{theorem}

The weighted projective space
\[
 X_*=\PP(1,1,12,28,42),\qquad
 (-K_{X_*})^4=\frac{84^4}{12\cdot28\cdot42}=3528
\]
satisfies all the hypotheses of Theorem~\ref{thm:canonical}.
Balletti--Kasprzyk--Nill
\cite[Corollary 1.3]{BKN} proved that \(3528\) is the sharp
upper bound for toric canonical Fano fourfolds, with equality holding
only for \(X_*\). Theorem~\ref{thm:canonical} establishes the same optimal constant
and equality classification in our class without a
toric assumption. To the best of our knowledge, this is the
first sharp upper bound with equality classification for
this canonical class beyond the toric setting.

We also obtain two further consequences under the same hypotheses. Every rank-three coherent subsheaf
\(\mathcal A\subset\TT_X\) satisfies
\[
 c_1(\mathcal A)(-K_X)^3\le\frac{41}{42}(-K_X)^4,
\]
and the coefficient \(41/42\) is sharp
(Corollary~\ref{cor:sharpdeterminant}).
Moreover, \((-K_X)^4\ge3130\) already implies
\(X\simeq X_*\) (Corollary~\ref{cor:volumegap}).
Thus the maximal-volume member is separated from the remaining members of this class
by an explicit volume gap.

We also prove a sharp bound when the non-Cartier locus of a primitive Weil divisor is zero-dimensional. Write \(\Cl(X)/\mathrm{tors}=\ZZ A\), choosing the positive generator \(A\), and let \(q\) be determined by \(-K_X\qsim qA\).

\begin{theorem}\label{thm:finite}
As in the setting of Theorem~\ref{thm:canonical}, suppose that \(A\) is Cartier outside a finite set. Then
\[
                         (-K_X)^4\le1024.
\]
This bound is attained by \(\PP(1,1,1,1,4)\). In particular, it applies if \(X\) has only isolated canonical singularities.
\end{theorem}

Unlike the equality classifications in Theorems~\ref{thm:terminal} and \ref{thm:canonical}, we do not claim a classification of all equality cases in Theorem~\ref{thm:finite}.

\subsection*{The main ingredients of the proof}
Let \(H=-K_X\) and \(V=H^4\), and normalize tangent slopes by \(V\). By Liu's generic ampleness theorem for the canonical extension \cite[Proposition 2.9]{Liu}, the normalized Harder--Narasimhan (HN) slopes of \(\TT_X\), each repeated according to the rank of the corresponding quotient, are positive:
\[
                  \lambda_1\ge\lambda_2\ge\lambda_3\ge\lambda_4>0,
                  \qquad \sum_i\lambda_i=1.
\]
We use the principal-parts filtration of Lyu--Zhang \cite[Lemma 2.3]{LZ}. Its reflexive version for K-semistable \(\QQ\)-Fano varieties is discussed in \cite[Remark 2.8]{LZ}. Retaining the individual normalized Harder--Narasimhan slopes gives a finite section estimate for integral \(\QQ\)-Cartier Weil divisors and an asymptotic volume bound. In the equality cases used below, the first-order part of the filtration also determines the dimension of the image of the associated rational map. These estimates hold in arbitrary dimension; the general statement is given in Theorem~\ref{thm:jets}. In the fourfold setting, for any integral \(\QQ\)-Cartier Weil divisor $D$, they read
\[
 V\le \frac{625}{256\prod_i\lambda_i},
 \qquad
 \J_\lambda\bigl(h^0(X,D)\bigr)\le t\,h^0(X,D)
 \quad(D\equiv tH),
\]
where \(\J_\lambda(N)\) is the sum of the weights of \(N\) monomials of smallest weight.

For a terminal \(X\), we prove that
\(\lambda_1\le1/3\) and \(\lambda_1+\lambda_2<2/3\).
The strict inequality follows from a surface argument on the relative terminal model of a general leaf. Kawamata's Cartierness theorem in dimension three then gives \(q^3\mid V\). These constraints leave only \(686\) and \(729\) above \(648\); primitive Weil jets and an elementary pencil obstruction exclude both candidates. When the volume is \(648\), divisibility of the determinant obtained from projective flatness determines the Weil index, and equality in the jet estimate for \(2A\) proves that the section ring is a polynomial ring generated in weights \((1,1,1,1,2)\).

For a canonical \(X\), we study the pencil defined by a nonzero proper term of the Harder--Narasimhan filtration of \(\TT_X\) of rank three. Write \(\tau=\lambda_4\), let \(D\) be its general
prime member, and let \(G\) be the normalization of \(D\), with
\(L=H|_G\) and \(\ell=L^3\). If the nonreduced divisorial
fibers have multiplicities \(m_i\), set
\[
 \beta=2-\sum_i(1-1/m_i),\qquad \eta=2/\beta.
\]
The multiplicities of the divisorial fibers enter both the degree of a general leaf and the jet estimate after base change through the same factor \(\eta\):
\[
 \ell=\frac{\eta\tau V}{2},
 \qquad \eta V\le\frac{625}{256\prod_i\lambda_i},
\]
where the asymptotic jet bound on the cover is given by Proposition~\ref{prop:normaljets}. When \(\tau\ge 1/32\), the low-rank slope bounds and majorization already force \(\eta V<3528\) (Proposition~\ref{prop:canonicalsharp}). Together with the low-rank slope bounds, these estimates imply \(\tau\ell<2\) when \(\tau<1/32\). Comparing the ambient discrepancy formula with the log-leaf formula then shows that the normalized general leaf is a canonical Gorenstein Fano threefold of Weil index \(\ell\):
\[
 K_G+L\sim R_G\ge0,\qquad
 \tau\ell=L^2R_G+L^2\Delta_G,
\]
where \(R_G\) is integral and every nonzero coefficient of
\(\Delta_G\) is at least one, so \(\tau\ell<2\) forces \(R_G=0\) and
\(\Delta_G=bB\) with \(L^2B=1\).
Liu's index bound \cite[Theorem 1.3]{LiuIndex} gives \(\ell\le42\), from which the bound
\[
 \eta V=\frac{2\ell^2}{b}\le2\cdot42^2=3528
\]
follows. In the equality case, we first determine the section ring of the leaf \(G\) from the plane \(B\simeq\PP^2\) with \(L|_B\simeq\OO_{\PP^2}(1)\), its transverse singular curves, and the quotient of \(R(G,B)\) by the defining section of \(B\), which gives \(G\simeq\PP(1,6,14,21)\). We then recover the section ring of \(X\) by ordinary adjunction: \(D\) is normal with different \(B/2\), and
\[
 R(X,D)/(s)\simeq\CC[t_1,v_{12},w_{28},z_{42}],
\]
where subscripts denote degrees. Lifting these generators recovers the weighted polynomial ring of \(\PP(1,1,12,28,42)\).

Finally, when \(\dim\NCart(A)\le 0\), the Weil Riemann--Roch recurrence together with finite section estimates and equality arguments gives the sharp bound
\(1024\) in Theorem~\ref{thm:finite}.
Table~\ref{tab:results} summarizes the volume bounds across the different classes; all
numerical comparisons are understood under the stated hypotheses.

\begin{table}[htbp]
\centering\small
\caption{Volume bounds for canonical Fano fourfolds. Here \(\mathcal C\)
denotes our standing assumptions: \(\QQ\)-factorial, Gorenstein,
and \(\rho=1\).}\label{tab:results}
\begin{tabularx}{\textwidth}{@{}p{.25\textwidth}Yp{.23\textwidth}@{}}
\toprule
Result & Additional hypotheses & Conclusion\\
\midrule
Birkar \cite{Birkar4} & None & General explicit bound\\
Theorem~\ref{thm:canonical} & \(\mathcal C\)
 & Sharp \(V\le3528\), with equality classification\\
Theorem~\ref{thm:finite} & \(\mathcal C\), \(\dim\NCart(A)\le0\)
 & Sharp \(V\le1024\)\\
Theorem~\ref{thm:terminal} & \(\mathcal C\), terminal
 & Sharp \(V\le648\), with equality classification\\
Balletti--Kasprzyk--Nill \cite{BKN} & Toric & Sharp \(V\le3528\)\\
\bottomrule
\end{tabularx}
\end{table}

The structure of this paper is as follows. Sections~\ref{sec:prelim} and \ref{sec:jets} collect the preliminary sheaf tools and prove the jet estimates. Section~\ref{sec:leaves} establishes the slope inequalities on general leaves of low-rank foliations. Sections~\ref{sec:terminal} and
\ref{sec:equality} prove Theorem~\ref{thm:terminal} and classify its equality case. Section~\ref{sec:canonical} proves the canonical volume bound in Theorem~\ref{thm:canonical}, and Section~\ref{sec:canonical-equality} determines its equality case and the volume gap. Section~\ref{sec:finite} treats the case of a finite non-Cartier locus (Theorem~\ref{thm:finite}). Section~\ref{sec:discussion} discusses possible extensions and the main difficulties involved.

\subsection*{Acknowledgements}
The authors are grateful to their advisor, Meng Chen, for his encouragement and support. They also thank their fellow students for helpful discussions and friendship.

\subsection*{AI disclosure}
The main framework of this article was developed by the authors. The authors used ChatGPT as a research-assistance tool during the preparation of this manuscript. In particular, the authors discussed mathematical ideas with ChatGPT 5.6 and 6.0, and used it for English language editing. This paper was completely written by the authors, who take full responsibility for its accuracy and content.
\hfill

\section{Preliminaries}\label{sec:prelim}

\subsection{Notation and standard results}

We follow the standard conventions on singularities, discrepancies, and the negativity lemma in \cite{KM}. Unless otherwise
stated, \(X\) is a \(\QQ\)-factorial Gorenstein canonical Fano
fourfold with \(\rho(X)=1\). Denote
\[
                  H=-K_X,\qquad V=H^4,\qquad H\qsim qA.
\]
The divisor class group \(\Cl(X)\) is finitely generated, and its quotient modulo torsion has rank one; these are standard consequences of the Mori dream space property for \(\QQ\)-factorial varieties of Fano type \cite[Corollary 1.3.2]{BCHM} together with the vanishing of \(\Pic^0(X)\). We fix an integral Weil divisor \(A\) representing the positive generator of \(\Cl(X)/\mathrm{tors}\). Every nonzero effective integral Weil divisor has class \(aA\) modulo torsion for some integer \(a>0\). The integer \(q=q_{\QQ}(X)\) is the \emph{Weil index} of \(X\). Note that we do not initially replace the \(\QQ\)-linear equivalence \(H\qsim qA\) by linear equivalence.
We write \(\sim\), \(\sim_{\QQ}\), and \(\equiv\) for linear,
\(\QQ\)-linear, and numerical equivalence, respectively. Define
\[
 \NCart(A)=\{x\in X:\OO_X(A)\text{ is not locally free at }x\},
\]
which is a closed subset of codimension at least two. 

Table~\ref{tab:notation} summarizes the main notation used in our fourfold arguments.
\begin{table}[htbp]
\centering\small
\caption{Notation for the fourfold arguments.}\label{tab:notation}
\begin{tabularx}{\textwidth}{@{}lY@{}}
\toprule
Symbol & Meaning\\
\midrule
\(H,V\) & \(-K_X\) and its top self-intersection \(H^4\)\\
\(A,q\) & Positive primitive Weil class and \(H\sim_{\QQ}qA\)\\
\(\beta_j;\lambda_i\) & Distinct normalized tangent block slopes;
the normalized HN slopes\\
\(C,\Chat\) & Resolution and orbifold second Chern numbers against \(H^2\)\\
\(u,\tau\) & \(-K_{\FF}\sim_{\QQ}uH\) and \(\tau=1-u\)\\
\(L,R,\Delta\) & Leaf polarization, integral ambient discrepancy
divisor, and log-leaf boundary\\
\bottomrule
\end{tabularx}
\end{table}

For a nonzero torsion-free sheaf \(\mathcal G\), write
\[
               \mu_H(\mathcal G)=\frac{c_1(\mathcal G)H^3}{\rk\mathcal G}.
\]

Throughout, we encode the Harder--Narasimhan (HN) filtration of \(\TT_X\) by the vector \(\lambda=(\lambda_1,\lambda_2,\lambda_3,\lambda_4)\) of its normalized HN slopes, each repeated according to the rank of the corresponding quotient. Specifically, if $Q_j$ is an HN quotient of rank $r_j$, then its normalized slope $\mu_H(Q_j)/V$ is listed $r_j$ times. Writing $c_1(Q_j)\equiv c_j A$, we have
\begin{equation}\label{eq:blocks}
 c_j\in\ZZ_{>0},\qquad \sum_jr_j=4,\qquad\sum_jc_j=q,\qquad
 \lambda=\left(\frac{c_j}{qr_j}\text{ repeated }r_j\text{ times}\right)_j.
\end{equation}
We write \(\beta_j=c_j/(qr_j)\) for the distinct normalized block
slopes, which decrease strictly and satisfy \(\sum_jr_j\beta_j^2=\sum_i\lambda_i^2\). Positivity of the $\beta_j$ follows from \cite[Proposition 2.9]{Liu}, since
\(\TT_X\) is a quotient of the generically ample canonical extension
\begin{equation}\label{eq:extension}
                    0\longrightarrow\OO_X\longrightarrow\EE_X
                    \longrightarrow\TT_X\longrightarrow0.
\end{equation}
Its construction and its first two orbifold Chern classes are described in \cite[Section 4]{GKPflat} and \cite[Section 2.4]{Liu}. In particular, \(c_1(\EE_X)=H\) and
\(\widehat c_2(\EE_X)=\widehat c_2(X)\).
The same positive-block description applies to \(\EE_X\), with \(\sum r_j=5\), and we write \(\alpha:\EE_X\to\TT_X\) for the natural projection in \eqref{eq:extension}.

We say that the HN filtration has a term of rank \(r\) (or a \emph{break at rank \(r\)}) if \(\lambda_r>\lambda_{r+1}\). By a \emph{proper HN term} of a torsion-free sheaf, we mean a nonzero proper subsheaf appearing in its Harder--Narasimhan filtration.

We will use the semistability of reflexive tensor and symmetric powers as in \cite[Theorem 4.2 and Corollary 4.3]{GKPmov}.
The standard HN polygon inequalities are unaffected by taking reflexive hulls or removing a closed subset of codimension at least two. On a klt variety that is not \(\QQ\)-factorial, one may pass to a small \(\QQ\)-factorialization \cite[Corollary 1.4.3]{BCHM} and use the movable curve class defined by the pullback polarization; since degrees and subsheaves agree in codimension one, these semistability properties also hold on the finite covers used in Section~\ref{sec:canonical}.

\subsection{Chern inequalities and Cartier Riemann--Roch}
Let
\[
                    \Chat=\widehat c_2(X)H^2,\qquad C=c_2(X)H^2.
\]
Here \(C=c_2(X)H^2\) means precisely
\[
 C:=c_2(\widetilde X)\cdot(f^*H)^2,
\]
where \(f:\widetilde X\to X\) is a smooth projective resolution
inducing the minimal resolution of each transverse Du Val surface singularity
in codimension two.

Let us briefly recall the origin of these two numbers. The invariant $\Chat$ is the intersection number of the orbifold second Chern class $\widehat{c}_2(X)$ of \(\TT_X\) with \(H^2\). Because a canonical variety has quotient singularities in codimension two, $\widehat{c}_2(X)$ is defined on local quotient charts away from a subset of codimension at least three, as in \cite[Section 2.2]{Liu}. 

On the other hand, $C$ is computed on a resolution that minimally resolves each transverse Du Val
surface singularity in codimension two. This number is independent of the choice of such a resolution: the calculation in Lemma~\ref{lem:cartierRR}
identifies \(C\) with \(24\) times the quadratic coefficient of the
Hilbert polynomial \(\chi(X,\OO_X(mH))\) minus \(V\), and any two such resolutions differ only over subsets of codimension at least three, which do not affect this intersection number.   

By construction, $C$ is the second-Chern contribution that appears in the Cartier Riemann--Roch formula, while $\Chat$ is the orbifold Chern number that appears in the Bogomolov--Gieseker discriminant inequality. The two numbers differ only by codimension-two singularity contributions; when $X$ is terminal, it is smooth in codimension two and thus $C=\Chat$.

For a reflexive sheaf \(\mathcal G\) of rank \(r\), we write
\[
 \Disc(\mathcal G):=2r\widehat c_2(\mathcal G)-(r-1)c_1(\mathcal G)^2
\]
for its orbifold discriminant; \(\Delta\) is reserved for leaf boundaries.

\begin{lemma}\label{lem:chern}
One has
\begin{equation}\label{eq:chern}
                  C\ge\Chat\ge\frac{1-\sum_i\lambda_i^2}{2}\,V.
\end{equation}
If \(\nu_1,\ldots,\nu_5\) are the normalized HN slopes of
\(\EE_X\), the same inequality holds with \(\sum_i\nu_i^2\) in place
of \(\sum_i\lambda_i^2\).
Moreover, if equality holds in the corresponding HN discriminant inequality, each reflexive HN quotient has zero orbifold discriminant against \(H^2\).
\end{lemma}

\begin{proof}
    \emph{Comparison of ordinary and orbifold Chern numbers.}
In codimension two, canonical singularities are transversely Du Val. Cutting \(X\) by two general members of \(|mH|\) for \(m\gg 0\) reduces the comparison of $C$ and $\Chat$ to the surface case. At each
Du Val point \(\CC^2/G\), the minimal resolution contributes the Euler number of the exceptional fiber, whereas the orbifold contribution is \(1/|G|\). Their difference is nonnegative (see also \cite[Theorem 4.6]{JLL} and \cite[proof of Theorem 4.5]{Liu}), so \(C\ge\Chat\). 

\emph{The HN discriminant comparison.}
Let \(Q_j\) be an HN quotient of \(\TT_X\) and \(G_j=Q_j^{**}\) be its reflexive hull of rank \(r_j\). The orbifold Bogomolov--Gieseker inequality
\cite[Section 2.2, equation (2.1)]{Liu} gives
\[
          \widehat{\operatorname{ch}}_2(G_j)H^2
          \le \frac{c_1(G_j)^2H^2}{2r_j}.
\]
We pass to a local smooth quotient chart and then to a general transverse surface $S$ to track the sign of the codimension-two correction. Reflexive pullbacks of the filtration terms and their reflexive quotients give vector bundles \(E_j'\) and \(G_j'\) on a smooth surface, fitting at each step into an exact sequence
\[
 0\longrightarrow E_{j-1}'\longrightarrow E_j'
 \longrightarrow G_j'\longrightarrow T_j'\longrightarrow0,
\]
where \(T_j'\) has finite length because $E'_j/E'_{j-1}$ agrees with $G'_j$ in codimension one on $S$, and the kernel of \(E_j'\to G_j'\) is reflexive and agrees with \(E_{j-1}'\) outside finitely many points. Taking second Chern characters gives
\[
 \operatorname{ch}_2(E_j')=
 \operatorname{ch}_2(E_{j-1}')+\operatorname{ch}_2(G_j')
 -\operatorname{length}(T_j').
\]
Dividing by the quotient-chart degrees and summing the correction terms yields a nonnegative orbifold defect \(\varepsilon \ge0\) such that
\[
\widehat{\operatorname{ch}}_2(\TT_X)H^2 = \sum_{j}\widehat{\operatorname{ch}}_2(G_j)H^2-\varepsilon.
\]
Since \(c_1(G_j)\equiv r_j\beta_jH\), we obtain the global numerical identity
\[
 \frac{V-2\Chat}{2}
 =\frac V2\sum_jr_j\beta_j^2
  -\sum_j\frac{\Disc(G_j)H^2}{2r_j}-\varepsilon
 \le\frac V2\sum_i\lambda_i^2,
\]
where each discriminant term is nonnegative by the orbifold
Bogomolov--Gieseker inequality. This proves \eqref{eq:chern}. Applying the same argument to \(\EE_X\) gives the second assertion. Since all discriminant and codimension-two defect terms in the sum are nonnegative, equality forces each discriminant term to vanish, proving the last assertion.
\end{proof}

\begin{lemma}\label{lem:cartierRR}
For every integer \(m\), we have the explicit Riemann--Roch formula
\begin{equation}\label{eq:cartierRR}
       \chi(X,\OO_X(mH))
       =1+\frac V{24}m^2(m+1)^2+\frac C{24}m(m+1).
\end{equation}
For \(m\ge0\), this equals \(h^0(X,mH)\). In particular,
\begin{equation}\label{eq:P1}
                       P_1:=h^0(X,H)=1+\frac V6+\frac C{12}.
\end{equation}
\end{lemma}

\begin{proof}
Let \(f:\widetilde X\to X\) be a resolution minimal in codimension two, and set \(\widetilde H=f^*H\).
Since canonical singularities are rational \cite[Theorem 5.22]{KM}, the Euler characteristics on
\(X\) and \(\widetilde X\) coincide: \(P(m):=\chi(X,\OO_X(mH))=\chi(\widetilde X,\OO_{\widetilde X}(m\widetilde H))\) for \(m\in \mathbb Z\).
In \(K_{\widetilde X}=-\widetilde H+E\), every component of \(E\)
has center of codimension at least three because the resolution is crepant in codimension two, so
\(\widetilde H^3E=\widetilde H^2E^2=0\). A direct computation via
Hirzebruch--Riemann--Roch \cite[Chapter 15]{Fulton} gives the coefficients \(V/24,V/12,(V+C)/24\) of \(m^4,m^3,m^2\).
Serre duality gives \(P(m)=P(-m-1)\), and \(P(0)=1\). These relations determine the remaining coefficients uniquely and yield \eqref{eq:cartierRR}. Since \(mH-K_X=(m+1)H\) is ample for \(m\ge0\), the integral \(\QQ\)-Cartier Weil-divisor form of Kawamata--Viehweg vanishing \cite[Theorem 1.2]{FujinoKV} gives $\chi(X,\OO_X(mH))=h^0(X,\OO_X(mH))$ for all $m\ge 0$.
\end{proof}

\subsection{Harder--Narasimhan terms and foliations}

We briefly recall the foliation terminology used below. A \textbf{foliation} on a normal variety $X$ is a saturated coherent subsheaf $\FF \subset \TT_X$ that is closed under the Lie bracket; equivalently, the induced morphism $\wedge^2\FF\rightarrow \TT_X/\FF$ is zero. On the smooth locus, $\FF$ corresponds to an integrable distribution of tangent subspaces. The foliation is \textbf{algebraically integrable} if the leaf through a general point is Zariski open in an algebraic subvariety of \(X\); we refer to the normalization of the closure of such a leaf as a \emph{general leaf}.

\begin{lemma}\label{lem:integrable}
Every proper HN term of \(\TT_X\) gives an algebraically integrable
foliation with rationally connected general leaf.
If \(\mathcal B\) is a proper HN term of \(\EE_X\), of rank
\(r\le4\), and \(c_1(\mathcal B)\equiv bH\), then we have
\begin{equation}\label{eq:anchorpolygon}
                          b\le\sum_{i=1}^r\lambda_i.
\end{equation}
\end{lemma}

\begin{proof}
For a proper HN term \(\FF\subset \TT_X\), positivity of $\FF$ and the strict slope inequality $\mu_{\min}(\FF)>\mu_{\max}(\TT_X/\FF)$ give
\[
  \mu_{\min}(\wedge^2\FF)\ge2\mu_{\min}(\FF)
            >\mu_{\max}(\TT_X/\FF),
\]
so \(\operatorname{Hom}(\wedge^2\FF,\TT_X/\FF)=0\).
Thus \(\FF\) is closed under the Lie bracket and defines a foliation. Pulling back to a resolution and computing slopes against the movable curve class of the pullback of \(H^3\), a general representing curve avoids the exceptional locus, so the minimal slope remains positive. Campana--P\u{a}un's criterion
\cite[Theorem 1.1]{CP} then gives algebraic integrability and rational connectedness of the general leaf. 

Since \(\mathcal B\) is saturated in \(\EE_X\), its intersection with
the kernel \(\OO_X\) of the natural projection \(\alpha:\EE_X\to\TT_X\) in \eqref{eq:extension} is either zero or
\(\OO_X\). In the first case, \(\mathcal B\) injects into \(\TT_X\), and the HN polygon upper bound gives
\[
bV=c_1(\mathcal{B})H^3\leq V(\lambda_1+\dots+\lambda_r).
\]
In the second case, its image in \(\TT_X\) has rank \(r-1\) and the same first Chern
class; since all $\lambda_i>0$, we have
\[
b\leq \lambda_1+\dots+\lambda_{r-1}< \lambda_1+\dots+\lambda_r.
\]
Thus \eqref{eq:anchorpolygon} holds in both cases.
\end{proof}

\subsection{Divisorial sheaves and ordinary adjunction}

We record two local facts used in the canonical equality case. For a
prime divisor \(D\subset X\), we write \(\operatorname{Diff}_{D^\nu}(0)\)
for the ordinary different on its normalization \(D^\nu\), so that ordinary
adjunction reads
\[
 K_{D^\nu}+\operatorname{Diff}_{D^\nu}(0)
 \qsim \nu^*((K_X+D)|_D).
\]
Note that \(\operatorname{Diff}_{D^\nu}(0)\) is distinct from the log-leaf boundary \(\Delta_G\) of a foliation leaf.

\begin{lemma}\label{lem:divisorialCM}
Let \(Y\) be a klt variety and \(M\) an integral \(\QQ\)-Cartier
Weil divisor. Then \(\OO_Y(M)\) is Cohen--Macaulay.
If \(D\subset Y\) is a prime \(\QQ\)-Cartier divisor, then \(D\)
is Cohen--Macaulay.
\end{lemma}

\begin{proof}
Working locally, choose a positive integer \(r\) such that
\(rM\sim0\). Since the order of a local defining equation
of \(rM\) along each prime divisor is divisible by \(r\), the normalized cyclic cover associated with this
trivialization is quasi-\'etale. The cover is
klt by \cite[Proposition 5.20]{KM}, and hence Cohen--Macaulay by
\cite[Theorem 5.22]{KM}. In characteristic zero, the eigensheaves in its finite pushforward
are the divisorial sheaves \(\OO_Y(iM)\) (with the sign depending
on the chosen character), which are direct summands. Because a system of parameters downstairs pulls back to a system of parameters
on the finite cover, each summand is a maximal Cohen--Macaulay
module, proving the first assertion (see also
\cite[Corollary 5.25]{KM}). For the second assertion, we apply the
depth lemma to
\[
 0\longrightarrow\OO_Y(-D)\longrightarrow\OO_Y
 \longrightarrow\OO_D\longrightarrow0.
\]
\end{proof}

\begin{lemma}\label{lem:ordinarydifferent}
Let \(Y\) have canonical Gorenstein singularities, let \(D\subset Y\)
be a reduced prime \(\QQ\)-Cartier divisor, and let
\(\nu:G\to D\) be its normalization. Suppose that \(K_G\) is
\(\QQ\)-Cartier. The coefficients of the ordinary different belong to
\[
 \{0\}\ \cup\
 \{1-1/r:r\in\ZZ,\ r\ge2\}\ \cup\ [1,\infty).
\]
In particular, its smallest positive coefficient is \(1/2\), and
every positive coefficient other than \(1/2\) is at least \(2/3\).
If all coefficients are less than one, then \(D\) is normal and
\((Y,D)\) is plt at the generic point of every prime divisor of \(D\).
At a coefficient \(1-1/r\), the corresponding transverse
singularity of \(Y\) is \(A_{r-1}\); a zero coefficient corresponds
to a smooth transverse point.
\end{lemma}

\begin{proof}
We compute at the generic point of a prime divisor
of \(G\). A general surface transverse to its image in \(Y\) is smooth
or Du Val. Passing to its smooth finite quasi-\'etale cover,
we obtain a reduced plane curve. Let \(\widetilde C\) be the
normalization of this curve, and let \(c\) be the coefficient of
its conductor divisor on a chosen branch; then \(c\) is a nonnegative integer that
vanishes if and only if the reduced curve is smooth at that
point. If \(r\) is the ramification index of the map from this
normalized branch to the normalized curve downstairs, adjunction
on the cover and the Hurwitz formula give
\[
 rd=c+(r-1),
 \qquad
 d=\frac{c+r-1}{r},
\]
where \(d\) is the different coefficient downstairs
\cite[Section 5.2]{KM}. If \(c=0\), then \(d=1-1/r\); if \(c\ge1\),
then \(d\ge1\). This also shows that the different is effective.

Suppose now that every coefficient is less than one. Then \(c=0\) at each such generic point, so the preimage
curve is a single smooth branch.
Its quotient is normal, meaning that \(D\) is regular in codimension one; since \(D\) is \(S_2\) by Lemma~\ref{lem:divisorialCM}, Serre's criterion
implies that \(D\) is normal. Let \(\Gamma\subset\mathrm{SL}_2(\CC)\)
be the group of the transverse Du Val cover. It preserves the
tangent line of the smooth branch, and its action on this line is
faithful because any finite-order element of \(\mathrm{SL}_2(\CC)\) with one eigenvalue equal to
one has both eigenvalues equal to one and is the identity.
Thus \(\Gamma\) is cyclic of order \(r\), and the transverse surface
singularity is \(A_{r-1}\). Since the smooth cover together with its smooth
boundary is plt, the same holds downstairs at the generic
point under consideration.
\end{proof}

\section{Jet estimates and the Harder--Narasimhan filtration}\label{sec:jets}

We establish a finite section estimate and an asymptotic volume bound in a form that also applies to
the primitive Weil divisor $A$. The principal-parts filtration is the one used by Lyu--Zhang \cite[Lemma 2.3]{LZ}; their Theorem~2.5 is polarized, and their Remark~2.8 discusses the reflexive extension to K-semistable \(\QQ\)-Fano varieties. Here we retain all individual HN slopes rather than replacing them by their average.

\begin{definition}\label{def:jetweight}
For a positive vector \(\lambda=(\lambda_1,\ldots,\lambda_n)\) and a monomial \(z^\alpha\) (\(\alpha\in\ZZ_{\ge 0}^n\)), write \(\operatorname{wt}_\lambda(z^\alpha)=\alpha\cdot\lambda\) for its \emph{weight}, and \(|\alpha|\) for its ordinary total degree. For an integer \(N\ge0\), let \(\J_\lambda(N)\) be the sum of the weights of \(N\)
monomials of smallest weight in
\[
                   \{\alpha\cdot\lambda:\alpha\in\ZZ_{\ge0}^n\},
\]
with distinct monomials counted separately, and \(\J_\lambda(0)=0\).
\end{definition}

\begin{theorem}[Jet estimates for Weil divisors]\label{thm:jets}
Let \(Z\) be a projective klt \(n\)-fold and \(P\) an ample
Cartier divisor. We compute tangent slopes with respect to \(P^{n-1}\),
and suppose the normalized HN slopes
\(\sigma_i=\mu_i(\TT_Z)/P^n\) are all positive. For an integral
\(\QQ\)-Cartier Weil divisor \(D\), denote
\(t_D=D\cdot P^{n-1}/P^n\). Then
\begin{equation}\label{eq:finitejet}
 \J_\sigma(h^0(Z,D))\le t_Dh^0(Z,D).
\end{equation}
In particular,
\begin{equation}\label{eq:asymjet}
 P^n\le\frac{((n+1)/n)^n}{\prod_i\sigma_i}.
\end{equation}
Here no proportionality between \(D\) and \(P\) is required.
The same conclusions hold for an ample \(\QQ\)-Cartier
polarization, with the indicated intersection-theoretic normalization.
\end{theorem}

\begin{proof}
Let \(W=H^0(Z,D)\) and \(N=\dim W\); we may assume \(N>0\). Over \(Z_{\reg}\), the sheaf \(\OO_Z(D)\) is a line bundle. Setting
\(K_0=W\otimes\OO_Z\), we define \(K_{j+1}\) for each \(j\ge0\) to be the kernel of the evaluation of jets through order \(j\) on \(Z_{\reg}\), extended as a saturated subsheaf of \(K_0\).
In codimension one, the symbol sequence for principal parts gives
\begin{equation}\label{eq:jetquotients}
            Q_j:=K_j/K_{j+1}\ \lhook\joinrel\longrightarrow\
            \Sym^j\Omega_Z^{[1]}\,\widehat\otimes\,\OO_Z(D),
\end{equation}
where \(\widehat\otimes\) denotes the reflexive tensor product, which does not affect codimension-one degrees.

At a general smooth point of \(Z\), the decreasing subspaces of \(W\) consisting of sections vanishing to order at least \(j\) have zero intersection by Krull's intersection
theorem. Since \(\dim W<\infty\), the filtration terminates, and setting \(b_j=\rk Q_j\) gives
\begin{equation}\label{eq:telescope}
                         \sum_j b_j=N,\qquad \sum_jc_1(Q_j)=c_1(K_0)=0.
\end{equation}

The HN filtration of $\TT_Z$ induces a filtration on each reflexive symmetric power $\Sym^j\Omega_Z^{[1]}$. Specifically, if the HN quotients $F_a$ of $\TT_Z$ have ranks $r_a$ and normalized slopes $\beta_a$, then after dualizing, the induced filtration on $\Sym^j\Omega_Z^{[1]}$ has graded factors $\widehat{\otimes}_a\Sym^{m_a}(F^*_a)$ with $\sum_{a} m_a =j$. By \cite[Theorem 4.2 and Corollary 4.3]{GKPmov}, these factors are slope semistable of normalized slope $-\sum_a m_a\beta_a$. Since each $\beta_a$ is repeated $r_a$ times in $\sigma=(\sigma_1,\ldots,\sigma_n)$, the rank of $\widehat{\otimes}_a\Sym^{m_a}(F^*_a)$ equals the number of multi-indices $\alpha\in \ZZ^n_{\geq 0}$ with $|\alpha|=j$ and weight $\alpha\cdot\sigma=\sum_a m_a\beta_a$. Reordering by slope and grouping equal slopes, the normalized HN slopes of $\Sym^j\Omega_Z^{[1]}\,\widehat\otimes\,\OO_Z(D)$ are precisely
\[
                        t_D-\alpha\cdot\sigma,\qquad |\alpha|=j,
\]
counted with multiplicities.

Let \(s_j(b)\) be the sum of the \(b\) smallest weights \(\alpha\cdot\sigma\) with
\(|\alpha|=j\). The rank-\(b_j\) HN polygon upper bound for \(Q_j\) in \eqref{eq:jetquotients} gives
\[
                \frac{c_1(Q_j)P^{n-1}}{P^n}
                      \le t_Db_j-s_j(b_j).
\]
Summing over \(j\) and using \eqref{eq:telescope} yields
\[
                  \J_\sigma(N)\le\sum_j s_j(b_j)\le t_DN,
\]
where the first inequality follows by dropping the restriction on the number of monomials chosen in each ordinary degree \(j\). This proves \eqref{eq:finitejet}.

For \eqref{eq:asymjet}, consider the simplex
\(\Delta_T:=\{(x_i)\in \mathbb R_{\ge 0}^n:\sum_i\sigma_ix_i\le T\}\), which has volume
\(T^n/(n!\prod_i\sigma_i)\). Comparing lattice points in \(\Delta_T\) with its volume gives
\[
 N_{\sigma}(T):=\#\{\alpha:\alpha\cdot\sigma\le T\}
       =\frac{T^n}{n!\prod_i\sigma_i}+O(T^{n-1}).
\]
Since the weight function $c(x)=\sum_i\sigma_ix_i$ is linear, its average over $\Delta_T$ equals its value at the barycenter. Under the rescaling $y_i=\sigma_ix_i/T$, $\Delta_T$ becomes the standard simplex $\{y_i\geq 0:\sum_i y_i\leq 1\}$ with barycenter $(\frac{1}{n+1},\dots,\frac{1}{n+1})$, so the average weight on $\Delta_T$ is $\frac{n}{n+1}T$.
Letting $T_N$ denote the weight scale of the $N$ monomials of smallest weight, we have
\[
T_N\sim\Bigl(n!N\prod_i\sigma_i\Bigr)^{\frac{1}{n}},
\]
and hence
\[
       \frac{\J_\sigma(N)}N
       \sim \frac n{n+1}
              \left(n!N\prod_i\sigma_i\right)^{1/n}.
\]
Applying \eqref{eq:finitejet} to \(D=mP\) with
\(N=h^0(Z,mP)\sim P^nm^n/n!\) and letting \(m\to\infty\) gives \eqref{eq:asymjet}. Finally, for an ample \(\QQ\)-Cartier polarization \(P\), choose \(a>0\) with \(aP\) Cartier; replacing \(P\) by \(aP\) divides each \(\sigma_i\) and \(t_D\) by \(a\), so homogeneity of \(\J_\sigma\) preserves \eqref{eq:finitejet} and the factors \(a^n\) cancel in \eqref{eq:asymjet}.
\end{proof}

\begin{proposition}[Asymptotic jet bound on a normal polarized variety]\label{prop:normaljets}
The volume inequality \eqref{eq:asymjet} holds for every normal
projective \(n\)-fold \(Z\) with an ample Cartier divisor \(P\)
and positive normalized tangent slopes \(\sigma_1,\ldots,\sigma_n\).
\end{proposition}

\begin{proof}
We apply the filtration \eqref{eq:jetquotients} to \(D=mP\).
Principal parts are formed on \(Z_{\reg}\), and the kernels
extend as saturated subsheaves of \(H^0(Z,mP)\otimes\OO_Z\), so
the rank and first Chern class identities in
\eqref{eq:telescope} hold on a normal variety.
Since reflexive tensor and symmetric powers preserve slope semistability on normal projective varieties
in characteristic zero
\cite[Theorem 3.1.4 and Corollary 3.2.10]{HL}, setting \(N_m=h^0(Z,mP)\) gives
\[
 \J_\sigma(N_m)\le mN_m.
\]
By Serre vanishing and the leading term of the Hilbert polynomial,
\[
 N_m=\frac{P^n}{n!}m^n+O(m^{n-1}).
\]
Combining this with the lattice-point asymptotics in the proof
of Theorem~\ref{thm:jets} yields
\[
 1\ge\lim_{m\to\infty}\frac{\J_\sigma(N_m)}{mN_m}
   =\frac n{n+1}\left(P^n\prod_i\sigma_i\right)^{1/n},
\]
which is \eqref{eq:asymjet}.
\end{proof}

\begin{remark}[Anticanonical normalization]\label{rem:anti-jet}
In our setting, we take \(P=H=-K_Z\) and write \(\lambda_i=\sigma_i\), so that
\(\sum_i\lambda_i=1\); if \(D\equiv tH\), the finite
estimate reads \(\J_\lambda(h^0(Z,D))\le t h^0(Z,D)\).
Neither this statement nor its volume consequence requires Gorensteinness. Lyu--Zhang's polarized theorem \cite[Theorem 2.5]{LZ} treats arbitrary polarizations when the tangent bundle is semistable; retaining the individual HN slopes \(\lambda_i\) allows us to handle the unstable cases below.
\end{remark}

\begin{definition}
    For nonincreasing vectors $v,\lambda\in \mathbb R^n$ with the same total sum, we say that $v$ \textbf{majorizes} $\lambda$ if for every $1\leq k<n$, 
    \[
    \sum_{i=1}^kv_i\geq \sum_{i=1}^k\lambda_i.
    \]
\end{definition}

\begin{lemma}[Majorization and equality information]\label{lem:jet-equality}
Suppose that \(v\) and \(\lambda\) are nonincreasing positive vectors
with the same sum, and that \(v\) majorizes \(\lambda\). Then
\begin{equation}\label{eq:majorization}
 \prod_i\lambda_i\ge\prod_i v_i,\qquad
 \J_\lambda(N)\ge\J_v(N),\qquad
 \sum_i\lambda_i^2\le\sum_i v_i^2.
\end{equation}
Suppose instead that the actual vector of normalized HN slopes is used in
\eqref{eq:finitejet} and equality holds. Then the ranks \(b_j\) in its proof select a set of \(N\) monomials of smallest weight.
If every such choice contains exactly \(b\) monomials of ordinary total degree one, then the image of the rational map defined by \(|D|\) has dimension \(b\).
\end{lemma}

\begin{proof}
The first and third inequalities follow from the concavity of \(\log\) and the convexity of the square. Moreover, \(\J_\lambda(N)\) is a
symmetric concave function of \(\lambda\), being the infimum of linear sums of weights over all sets of \(N\) distinct multi-indices, and is therefore Schur-concave, proving the middle inequality.

For the second assertion, in the proof of Theorem~\ref{thm:jets}, both \(\sum_js_j(b_j)-\J_\lambda(N)\) and each HN degree deficit are nonnegative, so equality forces all of them to vanish. Thus the ranks \(b_j\) select \(N\) monomials of smallest weight. Finally, choose a nonzero section \(s_0\in W\) near a general smooth point. The rank \(b_1\) of \(Q_1\) is the generic rank of the differentials \(d(s/s_0)\) for \(s\in W\), which in characteristic zero equals the transcendence degree of the field generated by these ratios, and hence the dimension of the image of \(\phi_{|D|}\). Thus $b=b_1=\rk Q_1=\mathrm{rank} (d\phi_{|D|}) = \dim \im|D|$.
\end{proof}

\subsection{The finite counting rule}\label{subsec:certificate}
For integer weights \(w=(w_1,\ldots,w_n)\), the coefficient of \(t^j\) in
\(\prod_i(1-t^{w_i})^{-1}\) counts monomials of weight \(j\).
At the first complete weight level \(j\) that cannot be included in full under the allowed average weight \(b\), let \(N\) and \(S\)
be the cumulative count and cumulative weight of all monomials of weight strictly less than \(j\). Then \(j>b\), and at most
\[
 \left\lfloor\frac{bN-S}{j-b}\right\rfloor
\]
monomials from level \(j\) may be added. Because the mean
weight up to level \(j\) tends to infinity with \(j\), this stopping rule always terminates and will be used in the finite section counts below.

\subsection{Finite covers and divisibility of determinants}

\begin{lemma}\label{lem:cover}
Let \(\pi:Y\to X\) be a finite quasi-\'etale cover of degree \(e\), where $Y$ is a normal projective variety. The normalized HN slopes of \(\TT_Y\) with respect to \(\pi^*H\) are given by \(\lambda\). Thus
\begin{equation}\label{eq:cover}
                         eV\le B(\lambda):=
                              \frac{625}{256\prod_i\lambda_i}.
\end{equation}
Therefore if \(B(\lambda)<2V\), there is no nontrivial such cover, \(\Cl(X)\) is torsion-free, and \(H\sim qA\).
\end{lemma}

\begin{proof}
We have \(K_Y=\pi^*K_X\) and \(\TT_Y=\pi^{[*]}\TT_X\) in codimension one; in particular, \(Y\) is Gorenstein
canonical Fano by \cite[Proposition 5.20]{KM}.

Slope semistability is preserved by reflexive pullback under \(\pi\): otherwise, passing to a Galois closure if necessary, the uniqueness of the maximal destabilizing subsheaf makes it Galois-invariant, so it descends over the \'etale locus and extends reflexively to a destabilizing subsheaf on \(X\), a contradiction. Since intersection numbers scale by \(e=\deg\pi\), the pulled-back HN
filtration is the HN filtration of \(\TT_Y\) and the vector \(\lambda\) of normalized HN slopes is unchanged. Applying Theorem~\ref{thm:jets} on \(Y\) gives \eqref{eq:cover}, noting that \(\rho(Y)=1\) is not needed there.

A torsion Weil class $L$ of exact order \(e>1\) defines a cyclic cover, \'etale in codimension one, by taking the normalization of its reflexive cyclic algebra, producing a normal projective variety $Y$ and a finite quasi-\'etale morphism $\pi: Y\to X$ of degree $e$. This is excluded by \eqref{eq:cover} when \(B(\lambda)<2V\). Thus \(\Cl(X)\) is torsion-free, and \(H\qsim qA\) implies \(H\sim qA\).
\end{proof}

The next lemma converts equality in the orbifold Bogomolov--Gieseker inequality into a divisibility statement in the class group.

\begin{lemma}\label{lem:flat}
Let \(X\) be projective klt with no nontrivial finite
quasi-\'etale cover. Fix an ample Cartier divisor \(H\).
If a reflexive sheaf \(\mathcal G\) of rank
\(r>1\) is slope semistable with respect to \(H\) and
\(\Disc(\mathcal G)H^{n-2}=0\), then the class
of \(\det\mathcal G\) in \(\Cl(X)/\mathrm{tors}\) is divisible
by \(r\).
\end{lemma}

\begin{proof}
\emph{Projective flatness.} By \cite[Proposition 1.6]{GKPflat},
\(\mathcal G|_{X_{\reg}}\) is locally free and projectively flat.
Recall that a projectively flat rank-\(r\) vector bundle on \(X_{\reg}\) is one whose associated projective bundle is flat; equivalently, after choosing a base point, parallel transport defines a projective monodromy representation
\[
\bar{\rho}: \pi_1(X_{\reg})\to \mathrm{PGL}_r(\CC).
\]
If \(\bar{\rho}\) is trivial, the only twist is scalar and \(\mathcal G|_{X_{\reg}}\simeq \mathcal L_0^{\oplus r}\) for a line bundle \(\mathcal L_0\) on \(X_{\reg}\).

\emph{Trivial projective monodromy.}
Since the fundamental group of a smooth complex quasi-projective variety is finitely generated (by finite semialgebraic triangulation \cite[Chapter 1]{Dimca}), the image of $\bar{\rho}$ is a finitely generated linear group.
By Mal'cev's residual finiteness theorem \cite{Malcev}, a nontrivial monodromy image has a nontrivial finite quotient. The associated connected finite \'etale cover \(U\to X_{\reg}\) is algebraic by the Riemann existence theorem \cite[Expos\'e XII]{SGA1}, and normalizing \(X\) in the function field of $U$ gives a nontrivial finite quasi-\'etale cover $\pi: Y\to X$, contradicting our hypothesis. Thus $\bar{\rho}$ is trivial.

\emph{Extension and determinant divisibility.}
Since $\bar{\rho}$ is trivial, parallel transport on $\mathcal{G}|_{X_\reg}$ is scalar, so the induced flat bundle $\End (\mathcal{G}|_{X_\reg})$ has trivial monodromy.
The standard matrix units $e_{ij}\in M_r(\CC)$ therefore define global holomorphic endomorphisms on $X_\reg$, which extend reflexively across the singular locus of codimension at least two and are algebraic on \(X\) by GAGA \cite{GAGA}. The diagonal idempotents $e_{ii}$ split \(\mathcal G\) into rank-one reflexive summands $\mathcal{L}_i$, and the relations $e_{ij}e_{ji}=e_{ii}$ and $e_{ji}e_{ij}=e_{jj}$ show that all $\mathcal{L}_i$ are isomorphic. Hence \(\mathcal G\simeq\mathcal L^{\oplus r}\) for a rank-one reflexive sheaf \(\mathcal L\), and \(\det\mathcal G\simeq\mathcal L^{[r]}\).
\end{proof}

\section{Geometry of low-rank foliations}\label{sec:leaves}

In this section, we establish explicit slope restrictions from foliations of rank one and two. In the terminal case, these bounds suffice to determine the sharp volume bound; in the canonical case, we combine them with the geometry of rank-three foliations in Section~\ref{sec:canonical}. 

\subsection{The family of leaves}
Let \(\FF\subset\TT_X\) be a proper HN term of rank \(r<4\), with
\(-K_{\FF}\qsim uH\). Since \(\mu_{H,\min}(\FF)>0\), Lemma~\ref{lem:integrable} shows that \(\FF\) defines an algebraically integrable foliation, and generic ampleness of \(\TT_X/\FF\)
gives \(0<u<1\). The normalized universal family of leaves gives a diagram
\[
                         X\xleftarrow{\ e\ }U\xrightarrow{\ p\ }T,
\]
where \(e\) is projective birational, \(p\) is equidimensional
with irreducible general fiber, and \(e\) restricts to the finite normalization on a general fiber \(G\) of \(p\). The construction and the log-leaf formula are recalled in \cite[Section 9]{JL} (following the numbering of \texttt{arXiv:2508.16364v3}):
\[
                 K_{e^{-1}\FF}+\Delta_U\qsim e^*K_{\FF},
                 \qquad \Delta_U\ge0,
\]
where \(\Delta_U\) is \(e\)-exceptional.

\begin{lemma}\label{lem:horizontal}
There exists a \(p\)-horizontal \(e\)-exceptional prime divisor. In particular, the normalized general leaf \(G\) has an effective boundary \(\Delta_G\) satisfying
\[
                   K_G+\Delta_G\qsim-uL_G,\qquad L_G=e|_G^*H,
\]
and some component of \(\Delta_G\) has coefficient at least one. Here the divisor \(L_G\) is ample Cartier.
\end{lemma}

\begin{proof}
Since \(X\) is \(\QQ\)-factorial, the exceptional locus of the projective birational morphism \(e\) is purely divisorial, so $e$ is an isomorphism away from the union of its exceptional divisors \(E_i\).

If all \(e\)-exceptional divisors were \(p\)-vertical, a general fiber \(G= p^{-1}(t)\) of \(p\) would lie in the isomorphism locus of \(e\). Choose
a general effective Cartier divisor $B$ on \(T\) avoiding $t$ and all images $p(E_i)$. Since $p^{-1}(B)$ contains no $e$-exceptional divisors, its pushforward $D:= e_*p^{-1}(B)$ is a nonzero effective divisor on \(X\) disjoint from $e(G)$. Because \(X\) is \(\QQ\)-factorial with $\rho(X)=1$, $D$ is an ample \(\QQ\)-Cartier divisor, contradicting $D\cap e(G)=\varnothing$.

For a \(p\)-horizontal \(e\)-exceptional prime divisor $E$, its coefficient in
\(\Delta_U\) is at least one by \cite[Proposition 3.5]{JLL}. Restricting to a general fiber \(G\) gives the stated boundary $\Delta_G$. Finally, since $e|_G: G\to e(G)$ is finite and $H|_{e(G)}$ is ample Cartier, \(L_G\) is ample Cartier.
\end{proof}

\begin{proposition}[Models of a general leaf]\label{prop:leafmodels}
Let \(G\), \(L_G\), and \(\Delta_G\) be as in
Lemma~\ref{lem:horizontal}.
\begin{enumerate}
\item On a simultaneous smooth resolution \(f:F\to G\) of the
family, \(L=f^*L_G\) is nef and big Cartier and
\begin{equation}\label{eq:ambient}
 K_F+L\sim R,\qquad R\ge0\text{ integral}.
\end{equation}
The divisor defined by \(K_F+\Delta=f^*(K_G+\Delta_G)\)
is effective away from the exceptional locus of \(f\); here the exceptional coefficients may be negative.
\item For \(r\le2\), a relative terminal model of the family
has a smooth general fiber \(F_{\rm term}\to G\), on which
\begin{equation}\label{eq:surfaceidentities}
 \begin{gathered}
 L\text{ is nef and big Cartier},\qquad K_{F_{\rm term}}+L\sim R,\\
 K_{F_{\rm term}}+\Delta\qsim-uL,
 \qquad R\ge0\text{ integral},\quad\Delta\ge0.
 \end{gathered}
\end{equation}
A component of \(\Delta\) has coefficient at least one and positive
integral \(L^{r-1}\)-degree. If the ambient variety \(X\) is
terminal, then additionally
\begin{equation}\label{eq:support}
 \Supp\Delta\subseteq\Supp R.
\end{equation}
\end{enumerate}

\end{proposition}

\begin{proof}
Choose a resolution \(\widetilde U\to U\) of the universal family and shrink the base \(T\) so that its smooth fibers resolve the normal general leaf \(G\). With \(\mu:\widetilde U\to X\), write
\[
 K_{\widetilde U}=\mu^*K_X+E_X.
\]
Since \(X\) is canonical and Gorenstein, \(E_X\) is an effective integral divisor. On a general fiber \(F\), the base
canonical divisor and the divisorial ramification of the family
restrict trivially, so adjunction gives
\(K_F+L\sim E_X|_F=:R\), proving (1). Any prime divisor on \(F\)
not exceptional over \(G\) has the same log-leaf coefficient as
its image in \(\Delta_G\ge 0\).

For (2), we use the relative terminal model
\(\pi:U_{\rm term}\to U\) of \cite[Lemma 9.2]{JL}, which is \(\QQ\)-factorial terminal with \(K_{U_{\rm term}}\) nef
over \(U\). Over a dense open subset of \(T\), the log-leaf formula gives
\[
 -\Delta_{U_{\rm term}}\text{ nef over }U,
 \qquad\pi_*\Delta_{U_{\rm term}}=\Delta_U\ge0,
\]
and the negativity lemma \cite[Lemma 3.39]{KM} implies that the \(p\)-horizontal
part of \(\Delta_{U_{\rm term}}\) is effective; note that this effectiveness holds on the relative minimal model \(U_{\rm term}\), rather than on an arbitrary resolution.

Since the terminal fourfold \(U_{\rm term}\) has singular locus of dimension at most one and the base \(T\) has dimension
\(4-r\ge2\), a general fiber \(F_{\rm term}\) avoids the singular locus and is smooth by generic
smoothness. Restricting the integral ambient discrepancy formula of \(U_{\rm term}\) to \(F_{\rm term}\) gives the effective integral divisor \(R\), and birationality of \(F_{\rm term}\to G\) makes \(L\) nef and big Cartier.
The component \(D\) of \(\Delta_G\) with coefficient at least one from
Lemma~\ref{lem:horizontal} has strict transform \(\widetilde D\) on \(F_{\rm term}\) with the same coefficient and
\[
 L^{r-1}\cdot\widetilde D=L_G^{r-1}\cdot D\in\ZZ_{>0}.
\]
Finally, every horizontal component of \(\Delta_{U_{\rm term}}\) is exceptional over \(X\).
When \(X\) is terminal, every such component has positive integral ambient discrepancy and thus lies in
\(\Supp R\), proving \eqref{eq:support}.
\end{proof}

In the rank-two surface arguments below, we denote the relative terminal general fiber \(F_{\rm term}\) from Proposition~\ref{prop:leafmodels}(2) simply by $F$.

\subsection{Canonical low-rank bounds}

\begin{lemma}\label{lem:surfacecoeff}
On a canonical del Pezzo surface \(S\), every coefficient of an
effective \(B\qsim-K_S\) is at most six. For \(K_S^2=7,8\)
the bound is four, and for \(K_S^2=9\) it is three.
\end{lemma}

\begin{proof}
Since \(-K_S\) is ample Cartier, if \(bD\le B\) with
\(D\) prime, then \(b\le b(-K_S)D\le K_S^2\), which proves the bound for \(K_S^2\le 6\). For \(K_S^2=7,8\), the minimal resolution \(\widetilde S\to S\) is a smooth weak del Pezzo surface.
By the classification of rational surfaces
\cite[Chapter V]{Beauville}, in degree eight \(\widetilde S\) is a Hirzebruch surface \(\mathbb F_e\) (\(0\le e\le2\)) with
\(-K_{\mathbb F_e}=2\sigma+(e+2)f\). Since the effective cone of \(\mathbb F_e\) is generated by
\(\sigma\) and \(f\), every prime curve in an effective anticanonical \(\QQ\)-divisor has coefficient at most four, and pulling \(B\) back crepantly to \(\widetilde S\) preserves the coefficient of its strict transform. In degree seven, \(\widetilde S\) is the
blowup of \(\PP^2\) at two (possibly infinitely near) points and has two distinct \((-1)\)-curves; contracting one distinct from the chosen component reduces to the degree-eight case. Finally, if $K_S^2=9$, then \(S\simeq\PP^2\), giving the bound three.
\end{proof}

\begin{proposition}\label{prop:canonicalslopes}
For a rank-one foliation \(\FF\) with \(\mu_{H,\min}(\FF)>0\), one has \(u\le1/2\). For a
rank-two foliation \(\FF\) with \(\mu_{H,\min}(\FF)>0\), one has \(u\le5/6\).
Consequently
\begin{equation}\label{eq:canonicalslopes}
                   \lambda_1\le\frac12,\qquad
                   \lambda_1+\lambda_2\le\frac56.
\end{equation}
If the HN filtration of \(\TT_X\) has no term of rank three, then
\(\lambda_4\ge1/12\).
\end{proposition}

\begin{proof}
In rank one, the normalized general leaf is \(\PP^1\), and
\eqref{eq:surfaceidentities} gives
\[
                    \deg L=2+\deg R\ge2,\qquad
                    u\deg L=2-\deg\Delta\le1,
\]
so \(u\leq 1/2\).

In rank two, suppose for contradiction that \(u>5/6\). By \cite[Theorem 9.3]{JL}, the relative terminal leaf model \(F\)
admits a birational morphism \(\phi:F\to\mathbb F_e\) with
\(e=1\) or \(2\). Push \eqref{eq:surfaceidentities} forward to
\(\mathbb F_e\), denoting the pushed-forward divisors by a subscript zero.
The divisor \(L_0\) is nef: at each point blowup,
\(L=\phi^*L_0-bE\) with \(b=LE\ge0\), so intersecting with strict transforms of curves downstairs shows that $L_0$ is nef.
Intersecting with a ruling fiber \(f\) (where $K_{\mathbb{F}_e}f=-2$, $(K_{\mathbb{F}_e}+L_0)f=R_0f$, and $(K_{\mathbb{F}_e}+\Delta_0)f=-uL_0f$) gives
\[
                            2\ge u(2+R_0f).
\]
Since \(u>5/6\) and \(R_0f\in\ZZ_{\ge 0}\), we have \(R_0f=0\).
Thus \(R_0\equiv af\) with \(a\in\ZZ_{\ge0}\), and $L_0\equiv 2\sigma+(e+2+a)f$. Effectivity of
\[
          \Delta_0\equiv -K_{\mathbb{F}_e}-uL_0 \equiv 2(1-u)\sigma+
                         ((e+2)(1-u)-ua)f
\]
forces \(a<1\), so \(a=0\), \(L_0\equiv-K_{\mathbb F_e}\),
and \(d:=L^2\le L_0^2=8\).
On \(F\), we now have
\[
                       (1-u)d=(R+\Delta)L<4/3,
                       \qquad \Delta L\ge1.
\]
Hence the nonnegative integer \(RL\) is zero, and $(1-u)d=\Delta L \geq 1$ with $u>5/6$ gives \(d=7\) or \(8\).
Since \(RL=0\), all components of \(R\) are exceptional over \(G\). Pushing
\eqref{eq:ambient} down to \(G\) gives \(K_G\sim-L_G\), so \(K_G\)
is Cartier, and the negativity lemma gives \(K_F-f^*K_G=R\ge 0\). Thus \(G\) is a
canonical Gorenstein del Pezzo surface of degree \(d\in\{7,8\}\). However, since
\(\Delta_G\qsim(1-u)(-K_G)\) has a component of coefficient at least one, $B:= \frac{1}{1-u}\Delta_G\sim_{\mathbb{Q}}-K_G$ has a component of coefficient at least $\frac{1}{1-u}>6$, contradicting Lemma~\ref{lem:surfacecoeff}.

Whenever the HN filtration of \(\TT_X\) has a term of rank one or two, the bounds just proved apply directly.
If there is no term of rank one, the first HN block has rank at least two, so $\lambda_1=\lambda_2$ and positivity of the slopes gives \(\lambda_1<1/2\). If there is no term of rank two, either the first block has
rank at least three (so $\lambda_1=\lambda_2=\lambda_3\le 1/3$ and \(\lambda_1+\lambda_2\le2/3\)), or a rank-one block is followed by a block of rank at least two (so $\lambda_2=\lambda_3\leq \frac{1-\lambda_1}{2}$ and
\(\lambda_1+\lambda_2\le(1+\lambda_1)/2\le3/4\)).
This proves \eqref{eq:canonicalslopes}.
Finally, if there is no term of rank three, the last HN block has rank two, three, or four, which gives
\(\lambda_4\ge1/12,1/6,\) or \(1/4\), respectively.
\end{proof}

\subsection{The strict terminal surface inequality}
In the terminal case, \eqref{eq:support} strengthens the rank-two surface argument.

\begin{lemma}\label{lem:strictsurface}
There is no smooth rational surface satisfying
\eqref{eq:surfaceidentities}--\eqref{eq:support}, with a boundary
coefficient at least one, for \(u=2/3\).
\end{lemma}

\begin{proof}
We track the integral multiplicities of \(R\) and the coefficients of \(\Delta\) throughout the argument.

\smallskip\noindent
\emph{Step 1: Contraction to a ruled surface.}
If \(F=\PP^2\), put \(r=\deg R\). The nonzero boundary and
the support condition \eqref{eq:support} give \(r\ge1\). The two canonical identities
give \(\deg\Delta=1-2r/3\le1/3\), so no component of \(\Delta\) can have
coefficient at least one.

Otherwise, choose a birational morphism \(F\to\mathbb F_e\) (stopping at \(\mathbb F_1\) if a minimal model is \(\PP^2\)). Pushing the divisors forward through the sequence
of point contractions preserves the canonical identities, the integrality and effectivity of \(R\), and the effectivity and support containment \(\Supp\Delta_0\subseteq\Supp R_0\). Nefness of \(L\) also descends:
at each contraction \(\beta:S'\to S\),
\(L'=\beta^*L-bE\) with \(b=L'E\ge0\), so for any curve
\(C\subset S\),
\(LC=L'\widetilde C+b\operatorname{mult}_pC\ge0\).

\smallskip\noindent
\emph{Step 2: Classification of the initial divisors.}
Let \(f\) be a ruling fiber and \(\sigma\) a minimal section on \(\mathbb F_e\) with
\(\sigma^2=-e\). If \(R_0f=0\), all components of \(R_0\) (and hence of \(\Delta_0\)) are
vertical, so \(\Delta_0f=0\), which would give
\(u=1\). Thus \(R_0f\ge1\), and
\[
 \frac23=\frac{2-\Delta_0f}{2+R_0f}
\]
forces \(R_0f=1\) and \(\Delta_0f=0\). Hence
\[
 R_0\equiv\sigma+af,\qquad
 L_0\equiv3\sigma+(e+2+a)f,\qquad
 \Delta_0\equiv\frac{e+2-2a}{3}f.
\]
Effectivity and nefness imply
\(a\in\ZZ_{\ge0}\), \(2a\le e+2\), and \(a\ge2e-2\), leaving five pairs $(e,a)$:
\[
 (e,a)=(0,0),(0,1),(1,0),(1,1),(2,2).
\]
When \(a=0\), \(R_0=\sigma\) has no vertical component whereas
\(\Delta_0\) is nonzero and vertical, contradicting \(\Supp\Delta_0\subseteq\Supp R_0\).

Since \(R_0f=1\), \(R_0\) has a unique horizontal prime component of multiplicity one, which is a section. Any section distinct from \(\sigma\)
has class \(\sigma+bf\) with \(b\ge e\), and the remaining vertical components of \(R_0\)
are fibers whose multiplicities sum to \(a-b\). Together with \(\Supp\Delta_0\subseteq\Supp R_0\), this gives the exhaustive list of actual divisors in Table~\ref{tab:surface}.

\begin{table}[htbp]
\centering\small
\caption{Initial configurations in the strict surface argument.
Each section denoted by \(C\) is smooth.}\label{tab:surface}
\begin{tabularx}{\textwidth}{@{}cYcl@{}}
\toprule
\((e,a)\)&\(R_0\)&\(\Delta_0\)&Degrees of $\Supp R_0$ against \(L_0\)\\
\midrule
\((0,1)\)&\(C\sim\sigma+f\)&\(0\)&\(L_0C=6\)\\
\((0,1)\)&\(\sigma+f_0\)&\(0\)&\(3,3\)\\
\((1,1)\)&\(\sigma+f_0\)&\(f_0/3\)&\(1,3\)\\
\((2,2)\)&\(C\sim\sigma+2f\)&\(0\)&\(L_0C=6\)\\
\((2,2)\)&\(\sigma+f_1+f_2\), \(f_1\ne f_2\)&\(0\)&\(0,3,3\)\\
\((2,2)\)&\(\sigma+2f_0\)&\(0\)&\(0,3\)\\
\bottomrule
\end{tabularx}
\end{table}
For example, when \((e,a)=(1,1)\), \(\Delta_0\equiv f/3\) forces a vertical component in its support, which excludes an
irreducible section \(R_0\sim\sigma+f\) and forces
\(R_0=\sigma+f_0\) and \(\Delta_0=f_0/3\).
In every configuration of Table~\ref{tab:surface}, the reduced support of $R_0$ has simple normal crossings with smooth irreducible components (in the final row, the fiber $f_0$ has multiplicity two in \(R_0\)).

\smallskip\noindent
\emph{Step 3: The first blowup.}
At a point blowup \(\beta:S'\to S\) centered at $p$ with exceptional curve \(E\), set
\[
 L'=\beta^*L-bE,\qquad b=L'E\in\ZZ_{\ge0},\qquad
 r=\operatorname{mult}_pR,\quad
 \delta=\operatorname{mult}_p\Delta.
\]
Let $R'$ and $\Delta'$ denote the corresponding divisors on $S'$ with $\beta_*R'=R$ and $\beta_*\Delta'=\Delta$, satisfying
\[
K_{S'}+L'\sim R',\quad K_{S'}+\Delta'\sim_{\mathbb{Q}}-\frac{2}{3}L'.
\]
Using \(K_{S'}=\beta^*K_S+E\) and writing $R'=\bar{R}+r'E$ and $\Delta'=\bar{\Delta}+\delta'E$, where $\bar{R}$ and $\bar{\Delta}$ are the strict transforms, the coefficients of \(E\) are given by
\begin{equation}\label{eq:blowuprule}
 r'=r+1-b,\qquad \delta'=\delta+\frac23b-1.
\end{equation}
Effectivity requires \(r',\delta'\ge 0\), support containment
requires \(r'\ge1\) whenever \(\delta'>0\), and nefness of \(L'\) requires \(b\le LC\) for each smooth
component \(C\) of \(\Supp R\) passing through \(p\).

When \(\Delta=0\), \(\delta'\ge0\) forces \(b\ge2\);
thus \(\delta'>0\), and support containment forces
\(b\le r\). A point lying on a single reduced component of \(R_0\) has
\(r=1\) and cannot be blown up. In the \((0,1)\) reducible case \(R_0=\sigma+f_0\),
the only admissible center is the intersection point \(\sigma\cap f_0\),
with \(r=b=2\), which yields
\(\Delta'=E/3\), \(\operatorname{coeff}_E R'=1\),
\(L'E=2\), and \(L'\)-degree one on both strict transforms \(\bar{\sigma},\bar{f}_0\).
In the \((2,2)\) cases with \(R_0\ge\sigma\), a center on \(\sigma\) is impossible since \(L_0\sigma=0\) forces \(b=0\). Thus neither \(C\sim\sigma+2f\) nor \(\sigma+f_1+f_2\) (\(f_1\neq f_2\)) admits a first
blowup. For \(R_0=\sigma+2f_0\), a center on \(f_0\setminus\sigma\) has
\(r=b=2\), producing \(\Delta'=E/3\), \(\operatorname{coeff}_E R'=1\), and \(L'E=2\), while \(\bar{f}_0\) retains multiplicity two and now has \(L'\)-degree one.

\smallskip\noindent
\emph{Step 4: The induction invariant.}
Once \(\Delta\) is nonzero (including the initial \((1,1)\)
case), the following conditions hold:
\[
 \begin{gathered}
 \Delta=\tfrac13 C_*,\qquad
 C_*\text{ smooth},\qquad \operatorname{coeff}_{C_*}R=1,\\
 \Supp R\text{ has simple normal crossings},\qquad
 LC\le1\text{ for every component }C\neq C_*\text{ of }\Supp R.
 \end{gathered}
\]
Consider the next blowup $\beta:S'\to S$ at a point $p\in S$, where $\delta\in\{0,1/3\}$.
If \(b=0\), \eqref{eq:blowuprule} gives \(\delta'<0\). If \(b\ge2\), the center \(p\) cannot lie on any component $D\neq C_*$ of $\Supp R$ (as nefness of $L'$ gives $1\ge LD=bE\bar{D}+L'\bar{D}\geq b\ge 2$), nor only on \(C_*\) (where \(r=1\), giving \(r'=2-b\le0\) while \(\delta'>0\)), nor outside \(\Supp R\) (where \(r=0\)).

Thus \(b=1\). Nonnegativity of \(\delta'\) then forces \(p\in C_*\) (\(\delta=1/3\)), giving \(\delta'=0\) and \(r'=r\ge 1\). The new exceptional curve $E$ satisfies $L'E=1$, while for every old component $C\neq C_*$, $L'\bar{C}=LC-E\bar{C}\leq LC\leq 1$ and its multiplicity in $R'$ is unchanged; the strict transform \(\bar{C}_*\) still has coefficient \(1/3\) in \(\Delta'\) and \(1\) in \(R'\). Blowing up a point on \(C_*\) preserves simple normal crossings of \(\Supp R\), so all invariant conditions persist under subsequent blowups. In particular, no coefficient of $\Delta$ can exceed \(1/3\), contradicting the existence of a component with coefficient at least one.
\end{proof}

\begin{proposition}\label{prop:terminalslopes}
If \(X\) is terminal, a rank-one foliation \(\FF\) with \(\mu_{H,\min}(\FF)>0\) has
\(u\le1/3\), and a rank-two foliation \(\FF\) with \(\mu_{H,\min}(\FF)>0\)
has \(u<2/3\). In particular,
\begin{equation}\label{eq:terminalslopes}
                     \lambda_1\le\frac13,\qquad
                     \lambda_1+\lambda_2<\frac23.
\end{equation}
\end{proposition}

\begin{proof}
For a rank-one foliation, the general leaf is \(\PP^1\) and \eqref{eq:support} gives \(\deg R\ge 1\), so \(\deg L= 2+\deg R\ge3\) and \(u\deg L\le1\), giving $u\le 1/3$. For a rank-two foliation, if the smooth leaf model \(F\) is \(\PP^2\) with line class $h$, then $R\sim rh$ with $r\in\ZZ_{\ge 1}$, $L\sim (r+3)h$, and $\Delta\equiv (3-u(r+3))h$; since $\deg \Delta\geq 1$, we get $u\leq \frac{2}{r+3}\leq 1/2$. Otherwise, \(F\) maps birationally to \(\phi:F\to \mathbb{F}_e\) as in Proposition~\ref{prop:canonicalslopes}, with $u(2+R_0f)=2-\Delta_0f\leq 2$. Here $R_0f\geq 1$: indeed, if $R_0f=0$, then $R_0$ has no horizontal component, so \eqref{eq:support} implies that $\Delta_0$ also has no horizontal component and $\Delta_0f=0$, which would give $u=1$, contradicting $0<u<1$. Thus $u\leq\frac{2}{2+R_0f}\leq2/3$, and Lemma~\ref{lem:strictsurface} excludes $u=2/3$.

If the first HN block of \(\TT_X\) has rank two, its slope is less than
\(1/3\); if it has rank at least three, positivity of the
remaining slopes gives the same bound. Thus \(\lambda_1\le1/3\).
If \(\lambda_1+\lambda_2=2/3\), then $\lambda_2\leq \lambda_1\leq 1/3$ forces \(\lambda_1=\lambda_2=1/3\). Since the four positive slopes sum to one, this block cannot have rank three or four, so \(\TT_X\) would have a rank-two HN term with \(u=2/3\), which is excluded above. Hence $\lambda_1+\lambda_2<2/3$.
\end{proof}

\section{Weil Riemann--Roch and the sharp terminal bound}\label{sec:terminal}

\subsection{Cartierness and divisibility}

\begin{lemma}[Divisibility from the non-Cartier locus]\label{lem:divisibility}
Let \(Z\) be a normal projective \(n\)-fold, let \(H\) be ample
Cartier, and let \(A\) be an integral \(\QQ\)-Cartier Weil divisor
with \(H\sim_{\QQ}qA\), \(q\in\ZZ_{>0}\). If
\(\dim\NCart(A)\le s\), where \(0\le s\le n-2\), then
\[
 q^{n-s-1}\mid H^n.
\]
If \(A\) is Cartier everywhere, then \(q^n\mid H^n\).
\end{lemma}

\begin{proof}
For every sufficiently large integer \(m\), the complete intersection \(Z_m\) of \(s+1\)
general members of \(|mH|\) avoids \(\NCart(A)\), so the top self-intersection of the restricted line bundle \(A|_{Z_m}\) is an integer:
\[
 (A|_{Z_m})^{n-s-1}=
 \frac{m^{s+1}H^n}{q^{n-s-1}}\in\ZZ.
\]
Since \(m^{s+1}\) and \((m+1)^{s+1}\) are coprime,
\(H^n/q^{n-s-1}\in\ZZ\), proving $q^{n-s-1}\mid H^n$. If \(A\) is Cartier everywhere, \(H^n=q^nA^n\) directly gives \(q^n\mid H^n\).
\end{proof}

\begin{lemma}\label{lem:finiteCartier}
If \(X\) is terminal, every integral \(\QQ\)-Cartier Weil divisor
is Cartier outside finitely many points. More generally, whenever
the primitive divisor \(A\) has this property, one has
\begin{equation}\label{eq:qcube}
                              q^3\mid V.
\end{equation}
\end{lemma}

\begin{proof}
A terminal variety is smooth in codimension two. At a general point of any one-dimensional singular stratum, a transverse three-dimensional local section is terminal and Gorenstein. By Kawamata's Cartierness theorem \cite[Lemma 5.1]{Kawamata}
(see also \cite[Proposition 2.1]{Prokhorov}), every \(\QQ\)-Cartier Weil divisor on a terminal Gorenstein threefold is Cartier. Since any divisorial sheaf on \(X\) is Cohen--Macaulay
\cite[Corollary 5.25]{KM}, its restriction to the transverse section is reflexive and coincides with the divisorial sheaf of the restricted Weil divisor.
Local freeness on the transverse section lifts to \(X\) by Nakayama's
lemma, proving Cartierness at the general point of every
one-dimensional singular stratum.

The divisibility \eqref{eq:qcube} follows from
Lemma~\ref{lem:divisibility} with \(n=4\) and \(s=0\).
\end{proof}

\begin{lemma}[Weil Riemann--Roch recurrence]\label{lem:Weilrecurrence}
Suppose that \(A\) is Cartier outside finitely many points and
that \(H\sim qA\) is an actual linear equivalence. For
\(F(n)=\chi(X,\OO_X(nA))\), one has, for every integer \(n\),
\begin{equation}\label{eq:Weilrecurrence}
 F(n)-F(n-q)=\frac{Vn^3}{6q^3}+\frac{Cn}{12q}.
\end{equation}
\end{lemma}

\begin{proof}
Let \(P(x)\) be the polynomial given in \eqref{eq:cartierRR}.
For each sufficiently large integer \(m>1\), a general member
\(D_m\in|mH|\) avoids the finite locus \(\NCart(A)\) and gives an exact sequence
\[
 0\longrightarrow\OO_X((n-mq)A)\longrightarrow\OO_X(nA)
 \longrightarrow\OO_{D_m}(nA)\longrightarrow0,
\]
since the defining section of \(D_m\) is a unit near \(\NCart(A)\) and the sheaves are line bundles elsewhere.

By Bertini's theorem on a log resolution, \(D_m\) is canonical, and adjunction gives
\[
 K_{D_m}\sim(m-1)qA|_{D_m},
\]
so \(D_m\) is a canonical Gorenstein threefold on which \(A|_{D_m}\) is Cartier. Thus
\[
 Q_m(n):=F(n)-F(n-mq)=\chi(D_m,\OO_{D_m}(nA))
\]
is a cubic polynomial in \(n\) with leading coefficient \(mV/(6q^3)\).
Serre duality on \(D_m\) gives
\(Q_m(n)=-Q_m((m-1)q-n)\), and
\(Q_m(0)=1-P(-m)\) since \(F(-mq)=P(-m)\).
Because an odd cubic polynomial about the nonzero center \((m-1)q/2\) is uniquely determined by its leading coefficient and its value at zero, and \(P(n/q)-P(n/q-m)\) shares these three properties, we have
\[
 Q_m(n)=P(n/q)-P(n/q-m).
\]
Subtracting \(Q_m(n-q)\) from \(Q_{m+1}(n)\) gives
\(F(n)-F(n-q)=P(n/q)-P(n/q-1)\), which equals the right-hand side of
\eqref{eq:Weilrecurrence}. Note that the fractional arguments here occur in a polynomial identity justified by Cartier restriction on \(D_m\);
no Riemann--Roch formula for a fractional Cartier multiple is assumed.
\end{proof}

\begin{corollary}[Primitive sections and Hilbert series]\label{lem:WeilRR}
Under the hypotheses of Lemma~\ref{lem:Weilrecurrence}, suppose
\(q\ge2\), and put \(k=V/q^3\), \(N=h^0(X,A)\). Then
\begin{equation}\label{eq:Weilsmall}
 h^0(X,sA)=sN+\frac{k(s^3-s)}6\quad(0<s<q),
 \qquad N=\frac{k}{6}+\frac{C}{12q}.
\end{equation}
The section ring has Hilbert series
\begin{equation}\label{eq:Weilseries}
 \sum_{n\ge0}h^0(X,nA)t^n
 =\frac{1+(N-4)t+(6-2N+k)t^2+(N-4)t^3+t^4}
 {(1-t)^4(1-t^q)}.
\end{equation}
\end{corollary}

\begin{proof}
For \(n>0\), the integral
\(\QQ\)-Cartier Weil-divisor form of Kawamata--Viehweg vanishing \cite[Theorem 1.2]{FujinoKV}
applies to both \(nA\) and \((n-q)A\) since their differences
from \(K_X\) are ample. When \(0<n<q\), \(h^0(X,(n-q)A)=0\), so Lemma~\ref{lem:Weilrecurrence} gives \eqref{eq:Weilsmall}, and for all \(n>0\),
\[
 h^0(nA)-h^0((n-q)A)=nN+\frac{k(n^3-n)}6.
\]
Multiplying by \(t^n\), summing over \(n\ge 1\), and using
\(\sum_{n\ge1}nt^n=t/(1-t)^2\) and
\(\sum_{n\ge1}(n^3-n)t^n=6t^2/(1-t)^4\) together with
\(h^0(\OO_X)=1\) gives $(1-t^q)\sum_{n\ge0}h^0(X,nA)t^n$ on the left-hand side; dividing by \(1-t^q\) yields \eqref{eq:Weilseries}.
\end{proof}

\subsection{Two consequences of primitivity}
The following observations do not require \(A\) to be Cartier and will also be used in the canonical case.

\begin{lemma}\label{lem:primitive}
Let \(A\) be a positive primitive Weil divisor on a normal projective
variety with \(\Cl(X)/\mathrm{tors}=\ZZ A\).
\begin{enumerate}
\item If \(h^0(A)\ge3\), the image of \(|A|\) is not a curve.
\item The kernel of
\(\Sym^2H^0(A)\to H^0(2A)\) contains no nonzero quadratic form
of rank at most four. If \(N=h^0(A)\ge3\), then
\begin{equation}\label{eq:quadraticgrowth}
                              h^0(2A)\ge4N-6.
\end{equation}
\end{enumerate}
\end{lemma}

\begin{proof}
Every effective member of \(|A|\) is a reduced prime divisor, so any two linearly independent sections in \(H^0(X,A)\) have distinct prime
zero divisors.

For (1), suppose $\dim \im|A|=1$. Resolve the rational map $\pi: \widehat{X}\to X$ and take the Stein factorization $\widehat{X}\to B\to \im\phi_{|A|}$
over a smooth curve \(B\). The moving part of $\pi^*|A|$ is the pullback of a basepoint-free linear system on \(B\) of dimension at least \(2\), hence of degree \(d\ge2\).
A general fiber $F$ of $\widehat{X}\to B$ pushes down to a nonzero effective divisor $D_F\equiv aA$ on $X$ with \(a\ge1\), and the fixed part pushes down to $E\equiv bA$ with \(b\ge0\), giving $A\equiv d D_F+E\equiv (ad+b)A$ and \(1=ad+b\ge 2\), a contradiction.

For (2), over \(\CC\), a nonzero quadratic relation of rank at most four
can be written as \(s_1s_2=s_3s_4\) or
\(s_1s_2=s_3^2\), allowing \(s_4=0\) in rank three.
Equality of the corresponding zero divisors would force
two linearly independent sections to share a prime component, which is impossible. Thus \(K:=\ker (\Sym^2H^0(A)\to H^0(2A))\) avoids the rank-at-most-four locus \(\Sigma_4=\{[Q]\in \PP(\Sym^2H^0(A)):\rk Q\leq 4\}\).
For \(N=3\), multiplication is injective and gives \eqref{eq:quadraticgrowth}.
For \(N\ge4\), $\dim \Sigma_4=\dim \mathrm{Gr}(N-4,N)+\binom{4+1}{2}-1=4N-7$, so $\codim \Sigma_4= N(N+1)/2-1-(4N-7)=(N-3)(N-4)/2$. Since $\PP(K)\cap \Sigma_4=\varnothing$, the projective dimension theorem gives $\dim K\leq (N-3)(N-4)/2$, and hence $h^0(X,2A)\ge N(N+1)/2- (N-3)(N-4)/2=4N-6$.
\end{proof}

\subsection{The finite terminal list}
For a positive vector \(v=(v_1,v_2,v_3,v_4)\) with \(\sum_i v_i=1\), write
\[
                            B(v)=\frac{625}{256\prod_i v_i}.
\]
Suppose throughout this subsection that \(X\) is terminal.
If \(q\ge6\), positivity and \eqref{eq:terminalslopes} imply
\(\lambda_4\ge1/q\): if the last HN block has rank \(r=1\), this follows from its positive integral degree;
for \(r=2\), the last two slopes sum to more than \(1/3\);
for \(r=3\), their common value is at least \(2/9\);
and for \(r=4\), it is \(1/4\).
Thus \(\lambda\) is majorized by
\[
                \left(\frac13,\frac13,\frac13-\frac1q,\frac1q\right).
\]
Lemma~\ref{lem:jet-equality}, Theorem~\ref{thm:jets}, and Lemma~\ref{lem:finiteCartier} yield
\[
                q^3\le V\le\frac{16875q^2}{256(q-3)},
\]
so \(q(q-3)\le16875/256\) and \(q\le9\).

For each \(1\le q\le9\), we enumerate the positive ordered compositions \((r_1,\ldots,r_\ell)\) of \(4\) and
\((c_1,\ldots,c_\ell)\) of \(q\) with strictly decreasing block slopes \(\beta_j=c_j/(qr_j)\) satisfying \eqref{eq:terminalslopes}. Table~\ref{tab:terminal} records the number of admissible vectors \(\lambda\), the maximum of \(B(\lambda)\), and the largest multiple of \(q^3\) not exceeding \(\max_\lambda B(\lambda)\).

\begin{table}[htbp]
\centering
\caption{Terminal HN slope vectors after the strict low-rank inequalities.
The last column also imposes \(q^3\mid V\).}\label{tab:terminal}
\begin{tabular}{@{}crcr@{}}
\toprule
\(q\)&Number of admissible \(\lambda\)&\(\max_{\lambda} B(\lambda)\)&Maximal volume satisfying $q^3\mid V$\\
\midrule
1&1&625&625\\
2&1&625&624\\
3&2&\(1366875/2048\)&648\\
4&1&625&576\\
5&3&\(390625/576\)&625\\
6&4&\(5625/8\)&648\\
7&4&\(1500625/2048\)&686\\
8&4&800&512\\
9&6&\(54675/64\)&729\\
\bottomrule
\end{tabular}
\end{table}

Only three pairs \((V,\lambda)\) satisfy \(V>648\) and \(q^3\mid V\):
\begin{equation}\label{eq:terminalcandidates}
\begin{array}{c|c|c}
q&V&\lambda\\ \hline
7&686&(2/7,2/7,2/7,1/7)\\
9&729&(8/27,8/27,8/27,1/9)\\
9&729&(1/3,5/18,5/18,1/9).
\end{array}
\end{equation}
In each case, \(B(\lambda)<855<2V\), so
Lemma~\ref{lem:cover} eliminates nontrivial finite
quasi-\'etale covers and gives \(H\sim qA\), making
Corollary~\ref{lem:WeilRR} applicable.

We now exclude these three cases. For \(q=7\), \eqref{eq:chern} gives \(C\ge252\), so
\[
                       N=\frac13+\frac C{84}\ge\frac{10}{3},
\]
and hence $N\ge 4$. Scaling \(\lambda\) by \(7\), the weights are
\((2,2,2,1)\) and the allowed average weight for \(A\) is \(1\).
The four smallest monomial weights sum to \(0+1+2+2=5>4\), so
Theorem~\ref{thm:jets} gives \(N\le3\), a contradiction.

For \(q=9\), the Chern bounds \eqref{eq:chern} for the two vectors \(\lambda\) are
\(C\ge264\) and \(C\ge1053/4\), giving
\[
                             N=\frac16+\frac C{108}>2,
\]
so $N\ge 3$. For both vectors \(\lambda\), the three monomials of smallest weight are \(1,z_4,z_4^2\) (where \(\operatorname{wt}_\lambda(z_4)=1/9\)), with total weight \(\J_\lambda(3)=1/9+2/9=1/3\), while every fourth monomial has weight greater than \(2/9\), so $\J_\lambda(4)>4/9$. Thus \eqref{eq:finitejet} forces \(N=3\) with equality, and since \(z_4\) is the only monomial of ordinary total degree one among \(1,z_4,z_4^2\), Lemma~\ref{lem:jet-equality} implies \(\dim\im|A|=1\), contradicting Lemma~\ref{lem:primitive}(1). This proves \(V\le648\).

\section{The equality case in the terminal bound}\label{sec:equality}

Throughout this section, assume \(X\) is terminal and \(V=648\).
Table~\ref{tab:terminal} and \(q^3\mid 648\) give \(q=3\) or \(6\).
Since \(B(\lambda)<2V\), Lemma~\ref{lem:cover} rules out nontrivial finite
quasi-\'etale covers and gives \(H\sim qA\).
Combining
\[
 \frac{V}{6q^3}+\frac{(1-\sum\lambda_i^2)V}{24q}
       \le N\le\max\{n:\J_\lambda(n)\le n/q\}
\]
leaves the following three possibilities:
\begin{equation}\label{eq:648data}
\begin{array}{c|c|r|r}
q&\lambda&N&C\\ \hline
3&(1/3,2/9,2/9,2/9)&11&252\\
6&(1/3,2/9,2/9,2/9)&4&252\\
6&(1/3,1/4,1/4,1/6)&4&252.
\end{array}
\end{equation}
(The vector \(\lambda=(5/18,5/18,5/18,1/6)\) at \(q=6\) is omitted because its Chern lower bound requires \(N\ge4\) whereas its finite jet bound gives \(N\le3\).)

\subsection{Excluding index three}
Suppose \(q=3\). Recall that $\EE_X$ in \eqref{eq:extension} has rank five and $c_1(\EE_X)=H\sim 3A$. If $Q_1,\dots,Q_\ell$ are its HN quotients, then
\[
r_j=\rk Q_j,\quad c_1(Q_j)\equiv c_jA,\quad \sum_jr_j=5,\quad \sum_jc_j=3,
\]
and the normalized block slopes $c_j/(3r_j)$ decrease strictly with $j$.

Since the positive integers $c_j$ sum to three, there are at most three blocks. Three blocks would have degrees $(1,1,1)$ and $r_1<r_2<r_3$, which is impossible since $\sum r_j=5$. Checking the two-block compositions $(1,2)$ and $(2,1)$ and the semistable case $(5;3)$ leaves five candidate HN types for $\EE_X$:
\begin{equation}\label{eq:E3}
 (r_j;c_j)=(1,4;1,2),\ (1,4;2,1),\ (2,3;2,1),\
                   (3,2;2,1),\ (5;3).
\end{equation}

The type \((1,4;2,1)\) is impossible: if the first HN block $\mathcal{B}\subset \EE_X$ has rank one and $c_1(\mathcal{B})\equiv cA$, then $\mathcal{B}$ has positive slope, so the natural projection $\alpha:\mathcal{B}\to \TT_X$ is generically injective. Saturating its image in $\TT_X$ gives a rank-one foliation with determinant class $dA$ (\(d\ge c\)), and Proposition~\ref{prop:terminalslopes} gives $d/3\le 1/3$, so $c=d=1$.

The type \((2,3;2,1)\) is also impossible. Let \(\mathcal B\subset\EE_X\)
be its first HN block, which is semistable of rank two and normalized slope
\(1/3\). If \(\mathcal B\) contains the kernel \(\OO_X\) of \(\alpha:\EE_X\to\TT_X\),
its rank-one image in \(\TT_X\) has determinant class at least
\(2A\), whose saturation in \(\TT_X\) contradicts the rank-one bound \(u\le 1/3\).
Otherwise \(\mathcal B \to \TT_X\) is generically injective, and the saturation \(\mathcal G\subset\TT_X\) of its image has rank two and $c_1(\mathcal{G})\equiv d A$ with $d\ge 2$. Then
\[
 \mu_{\min}(\mathcal G)/V\ge\mu_{\min}(\mathcal B)/V=1/3,\qquad
 c_1(\mathcal G)\equiv uH,\quad u=d/3\ge2/3.
\]
Since \(\TT_X/\mathcal G\) is a quotient of the generically ample sheaf \(\TT_X\) with $c_1(\TT_X/\mathcal G)\equiv (1-u)H$, we have
\[\mu_{\max}(\TT_X/\mathcal G)/V =\mu_1/V \le (\mu_1+\mu_2)/V=1-u\le1/3,\]
whereas $\mu(\wedge^2\mathcal G)=\mu_H(\det \mathcal G)= uV\ge 2V/3$. Thus \(\mathcal G\) is closed under the Lie bracket and defines a rank-two foliation with positive minimal slope and \(u\ge 2/3\), contradicting Proposition~\ref{prop:terminalslopes}.

For the remaining types \((1,4;1,2)\), \((3,2;2,1)\), and \((5;3)\) in \eqref{eq:E3}, the sums of the five
squared normalized slopes of \(\EE_X\) are
\[
                              \frac29,\qquad\frac{11}{54},
                              \qquad\frac15,
\]
so Lemma~\ref{lem:chern} for \(\EE_X\) gives
\(C\ge 252, 258, 1296/5\), respectively. Since $C=252$ by \eqref{eq:648data}, \(\EE_X\) must have type
\((1,4;1,2)\), and equality holds in the discriminant inequality for \(\EE_X\). By Lemma~\ref{lem:chern}, the rank-four reflexive HN quotient of \(\EE_X\) has zero orbifold discriminant and determinant class \(2A\). Lemma~\ref{lem:flat} makes \(2A\) divisible by four in \(\Cl(X)=\ZZ A\), a contradiction. Thus \(q=6\).

\subsection{Determination of the section ring}
Now \(q=6\), \(k=3\), and \(N=4\), so \eqref{eq:Weilseries} for $R:=R(X,A)=\bigoplus_{n\ge 0}H^0(X, nA)$ becomes
\begin{equation}\label{eq:ring648}
      H_R(t)= \sum_{n\ge0}h^0(nA)t^n
       =\frac{1+t^2+t^4}{(1-t)^4(1-t^6)}
       =\frac1{(1-t)^4(1-t^2)}.
\end{equation}
In particular, \(h^0(2A)=11\), and the right-hand side of \eqref{eq:ring648} is the Hilbert series of a polynomial ring on four generators of degree one and one generator of degree two.

We first show that \(\dim\im|2A|=4\). For the first \(q=6\) vector \(\lambda=(1/3,2/9,2/9,2/9)\) in \eqref{eq:648data}, multiplying the weights by \(9\) gives integer weights \((3,2,2,2)\) and an allowed
average weight of $9\cdot2\cdot(1/6)=3$ for \(2A\). The \(11\) monomials of smallest weight consist of \(1\), the four monomials of ordinary total degree one, and the six quadratic monomials in the three weight-two variables; their total weight is \(9+24=33=11\cdot3\), so equality holds in \eqref{eq:finitejet}.
For the second vector \(\lambda=(1/3,1/4,1/4,1/6)\), multiplying by $12$ gives integer weights \((4,3,3,2)\) and an allowed average weight of $12\cdot 2\cdot(1/6)=4$. There are eight monomials of weight less than \(6\) (of total weight \(26\)), and any three monomials of weight \(6\) complete a set of \(11\) monomials of smallest weight, with total weight \(26+18=44=11\cdot4\), so equality again holds.
For both vectors, all four monomials of ordinary total degree one have weight strictly smaller than the weight of the \(11\)-th monomial (\(4\) and \(6\), respectively), so every choice of \(11\) monomials of smallest weight contains all four linear monomials. By Lemma~\ref{lem:jet-equality}, $b_1=4$ and \(\dim\im|2A|=4\).

Choose a basis \(x_0,\ldots,x_3\) of \(R_1\). By Lemma~\ref{lem:primitive}(2), $\Sym^2R_1\to R_2$ is injective; since $\dim \Sym^2R_1=10$ and $\dim R_2=11$, we may choose \(y\in R_2\) so that $R_2=\Sym^2R_1\oplus\CC y$. This defines a graded homomorphism
\[
             \Phi: S=\CC[X_0,X_1,X_2,X_3,Y]\longrightarrow R,
              \qquad \deg X_i=1,\quad\deg Y=2,
\]
with image $R':=\im\Phi$ satisfying $R'_1=R_1$ and $R_2'=R_2$. Since the rational map defined by
\(R_2\) has image dimension four, $R'$ has Krull dimension five. Because $S$ is a five-dimensional polynomial domain, \(\Phi\) is injective, and comparing \eqref{eq:ring648} with the Hilbert series of \(S\) shows that \(\Phi\) is an isomorphism.

Since $H\sim 6A$ is ample Cartier, the sixth Veronese subring $R^{(6)}$ gives
\[
X\simeq \Proj R^{(6)}\simeq \Proj R =\PP(1,1,1,1,2).
\]
Conversely, \(\PP(1,1,1,1,2)\) is well-formed,
\(\QQ\)-factorial, of Picard number one, with Cartier anticanonical class
\(-K=\OO(6)\) of volume \(6^4/2=648\), and its unique quotient singularity
\(\frac12(1,1,1,1)\) has age two and is terminal by \cite[Section 4]{Reid}.
This proves Theorem~\ref{thm:terminal}.

\section{The canonical volume bound}\label{sec:canonical}

We now prove the sharp upper bound in Theorem~\ref{thm:canonical}.
Throughout this section, \(X\) is a \(\QQ\)-factorial Gorenstein canonical Fano fourfold with
\(\rho(X)=1\), and we use the low-rank slope bounds
\[
 \lambda_1\le\frac12,\qquad
 \lambda_1+\lambda_2\le\frac56
\]
from Proposition~\ref{prop:canonicalslopes}. If the HN filtration of \(\TT_X\) has no term of rank three, Proposition~\ref{prop:canonicalslopes} gives \(\lambda_4\ge 1/12\). We therefore study the rank-three case.

\subsection{The pencil and its divisorial fibers}

Suppose that \(\lambda_3>\lambda_4\), and let \(\FF\subset\TT_X\) be the
rank-three HN term, with \(\tau=\lambda_4\). By
Lemma~\ref{lem:integrable}, \(\FF\) has a rational first integral.
Taking its Stein factorization on a resolution gives a rational pencil with connected general fibers,
\[
 f:X\dashrightarrow\PP^1,
\]
where the base is \(\PP^1\) because \(X\) is rationally
connected \cite[Theorem 1]{ZhangRC}. Let \(D\) be the closure of
a general fiber of \(f\), let \(\nu:G\to D\) be its normalization, and put
\[
 L=\nu^*(H|_D),\qquad \ell=L^3.
\]
The divisor \(D\) is prime, and \(L\) is ample Cartier on \(G\), so \(\ell\in\ZZ_{>0}\).

\begin{lemma}\label{lem:pencilirreducible}
Every divisorial fiber of \(f\) has irreducible support. Write the
nonreduced divisorial fibers as \(m_iB_i\), where \(m_i\ge2\),
and set
\[
 \beta=2-\sum_i\left(1-\frac1{m_i}\right),\qquad
 \eta=\frac2\beta.
\]
Then \(\beta>0\), \(\eta\ge1\), and
\begin{equation}\label{eq:pencildegree}
 D\qsim\frac{\tau}{\beta}H,\qquad
 \ell=\frac{\eta\tau V}{2}.
\end{equation}
There are at most three nonreduced divisorial fibers.
\end{lemma}

\begin{proof}
Let \(B\) be a prime component of a divisorial fiber of \(f\). Since
\(\Cl(X)/\mathrm{tors}\) has rank one, there exist positive
integers \(a,b\) and \(h\in\CC(X)^*\) such that
\[
 \operatorname{div}(h)=aB-bD.
\]
Both \(B\) and \(D\) are \(\FF\)-invariant, so at a generic point
of either divisor (with local equation \(x=0\)), any local section
\(v\) of \(\FF\) satisfies \(v(x)\in(x)\).
Thus contraction with \(d\log h\) is regular in codimension one and defines a morphism
\[
 \FF\longrightarrow\OO_X.
\]
Since \(\mu_{H,\min}(\FF)>0\), this morphism is zero, so \(h\)
is constant along the leaves of \(\FF\). Because the general fibers of \(f\) are connected,
\(\CC(\PP^1)\) is algebraically closed in \(\CC(X)\), and
\(h\in\CC(\PP^1)^*\). Since the divisor of a rational function from \(\PP^1\)
contains every component of each fiber over its zeros and poles, \(B\) cannot be a proper component of a fiber, so every divisorial fiber has irreducible support.

Let \(N_{\FF}=(\TT_X/\FF)^{**}\) be the reflexive normal sheaf of \(\FF\). At a fiber of multiplicity \(m_i\), a local base parameter has order \(m_i\) and its differential vanishes to order \(m_i-1\), so saturating the differential of the pencil gives
\[
 N_{\FF}\sim 2D-\sum_i(m_i-1)B_i
           \qsim\beta D.
\]
On the other hand, the final rank-one HN quotient has normalized slope \(\tau\), and \(\rho(X)=1\) gives
\[
 N_{\FF}\qsim\tau H.
\]
Since \(D\) is ample and \(\tau>0\), we have \(\beta>0\), and intersecting
with \(H^3\) gives \eqref{eq:pencildegree}. Since each summand \(1-1/m_i\) in \(2-\beta\) is at least \(1/2\), there are at most three nonreduced divisorial fibers.
\end{proof}

\subsection{A jet estimate after base change}

The multiplicities of the divisorial fibers of the pencil defined by the rank-three HN term of \(\TT_X\) enter both the leaf degree
and the ambient jet estimate after base change through the same factor \(\eta\).

\begin{proposition}\label{prop:adjustedjets}
For the pencil defined by the rank-three HN term of \(\TT_X\), one has
\begin{equation}\label{eq:adjustedjets}
 \eta V\le B(\lambda)
       =\frac{625}{256\lambda_1\lambda_2\lambda_3\lambda_4}.
\end{equation}
\end{proposition}

\begin{proof}
First suppose that the pencil has at most two nonreduced fibers.
Choose coordinates \(0,\infty\) on \(\PP^1\) so that they include those
fibers (assigning multiplicity one to a reduced fiber if fewer than two fibers are nonreduced), and denote the two multiplicities by \(m,n\ge1\). Then
\[
 \beta=\frac1m+\frac1n,\qquad
 \eta=\frac{2mn}{m+n}.
\]
Put \(g=\lcm(m,n)\), and normalize \(X\) in the base extension
\(t=z^g\) to obtain a finite normal cover
\[
 \pi:Y\longrightarrow X,\qquad \deg\pi=g.
\]
Since \(\CC(X)/\CC(t)\) is regular, the base change \(Y\) is integral and its induced
pencil has connected general fiber.

At the generic point of
the fiber over \(t=0\), write \(t=ux^m\) in the local DVR, where
\(u\) is a unit. After an unramified extension extracting a \(g\)-th root
of \(u\), the normalization of \(z^g=ux^m\) has ramification
index \(g/m\) over \(X\), and \(z\) has order one along the
upstairs fiber. The calculation over \(t=\infty\) is the same, and away from \(0,\infty\) the base cover is unramified and the fibers are reduced; hence every divisorial fiber on \(Y\) is reduced.

Put \(P=\pi^*H\), and let \(\FF_Y\subset\TT_Y\) be the saturated foliation
of the upstairs pencil. Since the ramification is in the normal
direction to the fibers, a codimension-one chart \(x=y^e\) shows that
\[
 \FF_Y\simeq\pi^{[*]}\FF
 \quad\text{in codimension one}.
\]
Because reflexive pullback under a finite morphism preserves slope
semistability with respect to \(\pi^*H\)
\cite[Lemma 3.2.2]{HL}, the first three normalized HN
slopes of \(\FF_Y\) are \(\lambda_1,\lambda_2,\lambda_3\).
If \(D_Y\) is a general upstairs fiber, then
\(\pi^*D\sim gD_Y\), while the reducedness of all upstairs fibers gives
\(N_{\FF_Y}\sim2D_Y\), so \(N_{\FF_Y}\) has normalized slope
\[
 \tau'=\frac{\eta}{g}\tau,\qquad
 \frac{\eta}{g}=\frac{2\gcd(m,n)}{m+n}\le1.
\]
Since \(\lambda_3>\tau\ge\tau'\), the pulled-back HN filtration of \(\FF\) followed by \(N_{\FF_Y}\) forms
the HN filtration of \(\TT_Y\).

If there are three nonreduced fibers, their multiplicities
satisfy
\[
 \frac1{m_1}+\frac1{m_2}+\frac1{m_3}>1,
\]
so ordering them gives
\[
 (m_1,m_2,m_3)=(2,2,k),\ (2,3,3),\ (2,3,4),\ (2,3,5).
\]
Taking the base change by the classical Galois cover of \(\PP^1\) of degree \(g\) with branch orders \((m_1,m_2,m_3)\) associated with the dihedral, tetrahedral,
octahedral, or icosahedral group, where
\[
\begin{array}{c|cccc}
(m_1,m_2,m_3)&(2,2,k)&(2,3,3)&(2,3,4)&(2,3,5)\\ \hline
g&2k&12&24&60,
\end{array}
\]
we have \(g=2/\beta=\eta\) in each case. The same DVR calculation shows that the normalized cover \(\pi:Y\to X\) is quasi-\'etale with reduced
upstairs fibers, so the normalized HN slopes of \(\TT_Y\) are again
\(\lambda_1,\lambda_2,\lambda_3,\eta\tau/g\).

In both cases, all four slopes are positive and \(P^4=gV\), so
Proposition~\ref{prop:normaljets} gives
\[
 gV\le
 \frac{625}{256\lambda_1\lambda_2\lambda_3(\eta\tau/g)}
 =\frac g\eta B(\lambda),
\]
and dividing by \(g/\eta\) proves \eqref{eq:adjustedjets}.
\end{proof}

\subsection{The integral discrepancy on the leaf}

On the normalized general leaf \(G\) of the pencil, Lemma~\ref{lem:horizontal} gives
\begin{equation}\label{eq:canonicalleaf}
 K_G+\Delta_G\qsim-(1-\tau)L,\qquad \Delta_G\ne0,
\end{equation}
where every component of \(\Delta_G\) comes from a horizontal
evaluation-exceptional divisor and has coefficient at least one by
\cite[Proposition 3.5]{JLL}.

\begin{lemma}\label{lem:smallmass}
If \(\tau\ell<2\), then \(G\) is a canonical Gorenstein Fano
threefold and
\begin{equation}\label{eq:smallmass}
 \begin{gathered}
 K_G+L\sim0,\qquad
 \Delta_G=bB,\qquad
 L^2B=1,\qquad 1\le b<2,\qquad b=\tau\ell,\\
 L\qsim\ell B,\qquad
 q_{\QQ}(G)=\ell\in\{1,\ldots,22,24,30,42\}.
 \end{gathered}
\end{equation}
In particular,
\begin{equation}\label{eq:leafvolume}
 \tau\ge\frac1{42},\qquad
 \eta V=\frac{2\ell^2}{b}\le3528.
\end{equation}
\end{lemma}

\begin{proof}
Take a simultaneous resolution of \(X\) and the pencil, and let
\(a:F\to G\) be the induced birational morphism from a general smooth fiber \(F\). By
Proposition~\ref{prop:leafmodels}(1),
\[
 K_F+a^*L\sim R_F,\qquad R_F\ge0\text{ integral},
\]
and pushing down to \(G\) gives
\[
 K_G+L\sim R_G:=a_*R_F\ge0.
\]
Subtracting \eqref{eq:canonicalleaf} and intersecting with \(L^2\) yields
\begin{equation}\label{eq:massidentity}
 \tau\ell=L^2R_G+L^2\Delta_G.
\end{equation}
Since \(L\) is ample Cartier, \(L^2C\in\ZZ_{>0}\) for every prime divisor \(C\subset G\), so \(L^2\Delta_G\ge 1\). Because \(\tau\ell<2\) and \(R_G\ge 0\) is integral, we must have \(R_G=0\).
Thus \(K_G\sim-L\) is Cartier, and \(R_F\) is \(a\)-exceptional.
Applying the negativity
lemma to the \(a\)-exceptional divisor
\[
 (K_F-a^*K_G)-R_F\sim 0
\]
gives \(K_F-a^*K_G=R_F\ge0\), proving that \(G\) is a canonical Gorenstein Fano threefold.

The identity \eqref{eq:massidentity} also shows
that \(\Delta_G=bB\) for a single prime divisor \(B\subset G\) with
\(L^2B=1\) and \(1\le b=\tau\ell<2\). Combining \(K_G\sim -L\) with \eqref{eq:canonicalleaf} gives
\(bB\qsim\tau L\), and hence \(L\qsim\ell B\).
Conversely, if \(L\qsim q'C\) for an integral Weil divisor
\(C\), then
\[
 \ell=q'(L^2C),\qquad L^2C\in\ZZ_{>0},
\]
so \(q'\le\ell\) and \(q_{\QQ}(G)=\ell\).

Applying Liu's index theorem for canonical Gorenstein Fano threefolds
\cite[Theorem 1.3 and Corollary 4.8]{LiuIndex} (which requires no
\(\QQ\)-factoriality or Picard-number assumption) gives
the set of indices in \eqref{eq:smallmass}. Consequently,
\[
 \tau=\frac b\ell\ge\frac1{42},\qquad
 \eta V=\frac{2\ell}{\tau}
       =\frac{2\ell^2}{b}\le2\cdot42^2=3528.
\]
\end{proof}

\subsection{Completion of the volume bound}

\begin{proposition}\label{prop:canonicalsharp}
One has \(V\le3528\) and \(\lambda_4\ge1/42\).
If \(\TT_X\) has a rank-three HN term, then \(\eta V\le3528\) for the associated pencil;
if one of its divisorial fibers is nonreduced, then \(V\le2646\).
If \(V=3528\), the pencil exists and
\begin{equation}\label{eq:canonicalequalitydata}
 \begin{gathered}
 \eta=1,\qquad \tau=\frac1{42},\qquad \ell=42,\qquad
 \Delta_G=B,\qquad L^2B=1,\\
 K_G+L\sim0,\qquad L\qsim42B,\qquad H\qsim84D.
 \end{gathered}
\end{equation}
\end{proposition}

\begin{proof}
For \(0<t\le1/12\), the nonincreasing vector
\[
 v(t)=\left(\frac12,\frac13,\frac16-t,t\right)
\]
majorizes every admissible \(\lambda\) with \(\lambda_4\ge t\) by
\eqref{eq:canonicalslopes} and \(1-\lambda_4\le1-t\). Lemma~\ref{lem:jet-equality} gives
\begin{equation}\label{eq:canonicalmajorant}
 B(\lambda)\le B(v(t))=\frac{5625}{64t(1-6t)}.
\end{equation}
If the HN filtration of \(\TT_X\) has no term of rank three, Proposition~\ref{prop:canonicalslopes} gives
\(\lambda_4\ge1/12\), and hence
\[
 V\le B(v(1/12))=\frac{16875}{8}<3528.
\]
Suppose now that \(\TT_X\) has a rank-three HN term, with \(\tau=\lambda_4\). If \(\tau\ge1/32\),
Proposition~\ref{prop:adjustedjets} and
\eqref{eq:canonicalmajorant} give
\[
 \eta V\le B(v(1/32))=\frac{45000}{13}<3528,
 \qquad \tau>\frac1{42}.
\]
If \(\tau<1/32\), combining \eqref{eq:pencildegree}, \eqref{eq:adjustedjets}, and \eqref{eq:canonicalmajorant} gives
\[
 \tau\ell=\frac{\eta\tau^2V}{2}
 \le\frac{5625\tau}{128(1-6\tau)}
 <\frac{5625}{3328}<2,
\]
and Lemma~\ref{lem:smallmass} proves the stated bounds.

Any nonreduced fiber contributes at least \(1/2\) to \(2-\beta\),
giving \(\eta\ge4/3\) and
\(V\le3528\cdot(3/4)=2646\). Finally, when \(V=3528\),
only the case \(\tau<1/32\) is possible, and
\[
 3528\le\eta V=\frac{2\ell^2}{b}
             \le2\cdot42^2=3528
\]
forces \(\eta=1\), \(\ell=42\), and \(b=1\).
The identities in \eqref{eq:canonicalequalitydata} then follow
from Lemmas~\ref{lem:pencilirreducible} and
\ref{lem:smallmass}.
\end{proof}

\begin{corollary}\label{cor:sharpdeterminant}
Every rank-three coherent subsheaf \(\mathcal A\subset\TT_X\)
satisfies
\[
 c_1(\mathcal A)H^3\le\frac{41}{42}V.
\]
The coefficient \(41/42\) is optimal.
\end{corollary}

\begin{proof}
The HN polygon inequality and Proposition~\ref{prop:canonicalsharp} give
\[
 c_1(\mathcal A)H^3
 \le(\lambda_1+\lambda_2+\lambda_3)V
 =(1-\lambda_4)V\le\frac{41}{42}V.
\]
On \(\PP(1,1,12,28,42)\), the pencil defined by the two weight-one
coordinates has normal sheaf \(\OO(2)\) and tangent subsheaf of
determinant \(\OO(82)\), while \(H=\OO(84)\), attaining the coefficient \(41/42\).
\end{proof}

The weighted projective space \(\PP(1,1,12,28,42)\) is well-formed,
\(\QQ\)-factorial, of Picard number one, and Gorenstein canonical by
\cite[Section 4]{Reid} since every weight divides \(84\). Its anticanonical volume is
\[
 \frac{84^4}{1\cdot1\cdot12\cdot28\cdot42}=3528.
\]
We determine all equality cases in the next section.

\section{The equality case in the canonical bound}\label{sec:canonical-equality}

We first determine the threefold \(G\) appearing in
\eqref{eq:canonicalequalitydata}, and then use ordinary
adjunction and the divisorial section ring \(R(X,D)\) to recover \(X\).

\subsection{The threefold of index \(42\)}

\begin{proposition}\label{prop:leaf42}
Let \(G\) be a canonical Gorenstein Fano threefold. Put
\(L=-K_G\), and suppose that \(B\) is a prime Weil divisor with
\[
 L^3=42,\qquad L\qsim42B,\qquad L^2B=1.
\]
Then
\[
 (G,B)\simeq
 \bigl(\PP(1,6,14,21),\{x_1=0\}\bigr),
\]
where \(x_1\) is the weight-one coordinate. Moreover,
\[
 R(G,B):=\bigoplus_{m\ge0}H^0(G,\OO_G(mB))
       \simeq\CC[x_1,u_6,v_{14},w_{21}]
\]
as graded algebras; subscripts indicate degrees.
\end{proposition}

We prove Proposition~\ref{prop:leaf42} in several steps.

\begin{lemma}\label{lem:leafplane}
The group \(\Cl(G)\) is torsion-free, \(L\sim42B\), and
\[
 (B,L|_B)\simeq(\PP^2,\OO_{\PP^2}(1)).
\]
\end{lemma}

\begin{proof}
A torsion Weil class on \(G\) of exact order \(d>1\) defines a connected
cyclic quasi-\'etale cover \(G'\to G\) of degree \(d\), which is
a canonical Gorenstein Fano threefold of degree
\[
 (-K_{G'})^3=dL^3=42d\ge84,
\]
contradicting Prokhorov's bound \(72\)
\cite[Theorem 1.5]{ProkhorovDegree}. Thus \(\Cl(G)\) is
torsion-free and \(L\sim42B\).

By the classification of canonical Gorenstein Fano threefolds with nonempty anticanonical base locus
\cite[Theorem 1.1]{JR}, any such threefold has volume \(2,4,6\), or \(2m-2\)
with \(3\le m\le12\). Since \(L^3=42\), \(|L|\) is basepoint-free and defines a finite morphism \(\varphi\).
From
\[
 1=L^2B=\deg(\varphi|_B)\deg\varphi(B),
\]
we see that \(\varphi(B)\) is a plane \(\PP^2\) and \(\varphi|_B\) is finite birational, hence an isomorphism by the normality of \(\PP^2\).
\end{proof}

\begin{lemma}\label{lem:leafcurves}
There are exactly three singular curves along which \(B\) is
not Cartier. Their transverse types are \(A_1,A_2,A_6\);
they have \(L\)-degree one and are three distinct lines
\(L_2,L_3,L_7\) on \(B\simeq\PP^2\). At each of them, the
local class of \(B\) is a generator, up to sign. Furthermore,
\[
 h^0(G,B)=1,\qquad
 \operatorname{Diff}_B(0)=\frac12L_2+\frac23L_3+\frac67L_7.
\]
\end{lemma}

\begin{proof}
Let \(J\) be the smallest positive integer such that \(JB\)
is Cartier in codimension two. By \cite[Theorem 4.2]{JLL},
\[
 J\mid42,\qquad 42^2\mid JL^3=42J,
\]
so \(J=42\).

For each irreducible singular curve \(C\subset G\), let \(e_C\) be the
Euler number of the minimal resolution fiber of its transverse
Du Val singularity, \(g_C\) the order of its local
\(\mathrm{SL}_2(\CC)\)-subgroup, and \(j_C\) the exponent of its
local divisor class group. Since the local Cartier index of \(B\) along \(C\)
divides \(j_C\), we have
\[
 42\mid\lcm_C(j_C).
\]
Combining the Chern comparison formula of
\cite[Theorem 4.6]{JLL} with the inequality of
\cite[proof of Theorem 4.1]{LiuIndex} gives
\begin{equation}\label{eq:42chern}
 \sum_C\left(e_C-\frac1{g_C}\right)(L\cdot C)
 =24-\widehat c_2(G)L
 \le24-\frac{L^3}{4}=\frac{27}{2}.
\end{equation}
(Here the inequality of \cite[proof of Theorem 4.1]{LiuIndex} is applied on a small
\(\QQ\)-factorialization of \(G\), which is crepant, nef and big, of Weil index
\(42\), and isomorphic to \(G\) in codimension two since surface Du Val
singularities are \(\QQ\)-factorial.)

In \eqref{eq:42chern}, the factor
\(7\mid 42\) must come from an \(A_{r-1}\) curve \(C\), which contributes
\[
 \left(r-\frac1r\right)(L\cdot C).
\]
Since \(7\mid r\) with \(r\ge14\) or \(L\cdot C\ge 2\) exceeds \(27/2\), there is a unique \(A_6\) curve of \(L\)-degree one.
A \(D\)-type curve contributes at least \(5-1/8=39/8\), which together with the required factors \(3\) and \(7\) would contribute at least
\[
 \frac{39}{8}+\frac83+\frac{48}{7}>\frac{27}{2},
\]
while an \(E_6,E_7,\) or \(E_8\) curve contributes at least \(7-1/24\), and
\[
 7-\frac1{24}+\frac{48}{7}>\frac{27}{2}.
\]
Hence all transverse singularity types are of type \(A\).

Subtracting \(48/7\) from \(27/2\) leaves \(93/14\). The exponent \(j_C\) supplying
the factor \(3\) is \(3\) or \(6\). Exponent \(6\) leaves no
room for another curve; for exponent \(3\), the
exponent supplying the factor \(2\) is \(2\) or \(4\), where
exponent \(4\) leaves no room for another curve and
exponent \(2\) permits at most one additional exponent \(2\).
Counting a curve of \(L\)-degree \(d\) as \(d\) copies of its exponent, the possible lists of transverse exponents are
\begin{equation}\label{eq:42baskets}
 (6,7),\quad(2,3,7),\quad(3,4,7),\quad(2,2,3,7).
\end{equation}

Let \(k_C\in\{0,1,\ldots,j_C-1\}\) denote the local class of \(B\) along an
\(A_{j_C-1}\) curve \(C\). Since \(G\) is Gorenstein, a crepant terminalization of \(G\)
is Gorenstein and its isolated Reid basket is empty, so the singular Riemann--Roch formula of
\cite[Theorem 4.10]{JL} (applied with \(s=1\) and \(q=42\)) gives
\begin{equation}\label{eq:42RR}
 h^0(G,B)=2+\frac1{84}
 -\sum_C(L\cdot C)\frac{k_C(j_C-k_C)}{2j_C}.
\end{equation}
Along the \(A_6\) curve, put \(a=\min\{k_C,7-k_C\}\in\{1,2,3\}\).
For the other required local classes, the integrality of \eqref{eq:42RR} leads to the following short calculation, where the entries record \(84(h^0(G,B)-2)\):
\[
\begin{array}{c|c|rrr}
\text{exponents}&\text{other nonzero local classes}&a=1&a=2&a=3\\ \hline
(6,7)&k_6=\pm1&-70&-94&-106\\
(2,3,7)&k_2=1,\ k_3=\pm1&-84&-108&-120\\
(3,4,7)&k_3=\pm1,\ k_4=2&-105&-129&-141\\
(2,2,3,7)&k_2=k'_2=1,\ k_3=\pm1&-105&-129&-141
\end{array}
\]
Only the entry \(-84\) in the second row is divisible by \(84\).
(In the fourth row \((2,2,3,7)\), one of the two degree-one
\(A_1\) curves may have local class zero, reducing
\eqref{eq:42RR} to the second row; a single degree-two \(A_1\)
curve cannot do so since its two contributions have the same nonzero
local class.) Thus the curves along which \(B\) is
not Cartier have the stated types and local classes,
and \(h^0(G,B)=1\).

Because \(B\) is Cartier away from \(B\), each of these three curves is contained in \(B\), so by
Lemma~\ref{lem:leafplane} they are distinct lines \(L_2,L_3,L_7\) on \(B\simeq\PP^2\). Ordinary adjunction gives
\[
 (K_G+B)|_B\qsim K_B+\operatorname{Diff}_B(0),
\]
so against \(\OO_{\PP^2}(1)\),
\[
 \deg\operatorname{Diff}_B(0)
 =3-\frac{41}{42}=\frac{85}{42}.
\]
At a transverse \(A_{r-1}\) point where the local class of \(B\)
has order \(r\), the local calculation of Lemma~\ref{lem:ordinarydifferent} shows that the coefficient of \(\operatorname{Diff}_B(0)\) is at
least \(1-1/r\). Consequently,
\[
 \operatorname{Diff}_B(0)\ge
 \frac12L_2+\frac23L_3+\frac67L_7,
\]
and since the right-hand side has degree \(85/42\), equality holds.
In particular, \((G,B)\) is plt at the generic point of every
curve on \(B\).
\end{proof}

\subsection{The quotient of the threefold section ring}

Let \(R=R(G,B)\), let \(s\in R_1\) be a section defining
\(B\), and put \(S=R/(s)\).

\begin{lemma}\label{lem:leafquotientnormal}
The ring \(S\) is a normal graded domain, and
\[
 S^{(42)}
 =\bigoplus_{k\ge0}S_{42k}
 \simeq\bigoplus_{k\ge0}H^0(\PP^2,\OO_{\PP^2}(k)).
\]
In particular, \(\Proj S=B\).
\end{lemma}

\begin{proof}
Since \(B\) is an ample \(\QQ\)-Cartier integral Weil divisor, \(R\) is a finitely generated normal domain, and every \(\OO_G(mB)\)
is Cohen--Macaulay by Lemma~\ref{lem:divisorialCM}. For \(m>-42\),
divisorial Kawamata--Viehweg vanishing applies to
\[
 mB-K_G\sim(m+42)B,
\]
and for \(m\le-42\), Serre duality reduces the intermediate
cohomology to that of \((-42-m)B\), giving
\[
 H^i(G,\OO_G(mB))=0
 \qquad(i=1,2;\ m\in\ZZ).
\]
By \cite[Section 2]{Watanabe}, \(R\) is
Cohen--Macaulay, and hence so is \(S=R/(s)\).
Moreover, for any nonzero homogeneous element \(a\in R_m\), the coefficient of \(B\) in \(\operatorname{div}(a)+mB\) is positive if and only if \(s\mid a\); since these coefficients add under multiplication, \((s)\) is a prime ideal and \(S\) is a domain.

For each integer \(m\), the quotient sheaf
\[
 Q_m=\OO_G(mB)\big/s\OO_G((m-1)B)
\]
is annihilated by \(\OO_G(-B)\), so the depth lemma shows that \(Q_m\) is a rank-one Cohen--Macaulay sheaf—and hence a line bundle—on \(B\simeq\PP^2\). Cohomology vanishing gives
\[
 S_m=H^0(B,Q_m)\qquad(m\ge0)
\]
compatibly with multiplication, and \(42B\sim L\) gives
\[
 Q_{42k}\simeq\OO_B(kL),\qquad
 Q_{m+42k}\simeq Q_m\otimes\OO_B(kL),
\]
proving the statement for \(S^{(42)}\) and identifying \(\Proj S\) with \(B\).

To check regularity of \(\Spec S\) in codimension one away from the vertex (which has codimension three), note that the punctured section cone is locally obtained from the
local cyclic covers on which \(B\) becomes Cartier by adjoining an invertible
grading parameter and taking the finite cyclic quotient.
At the generic point of each curve on \(B\), the preimage of \(B\) on this local cyclic cover is smooth by
Lemma~\ref{lem:leafcurves}, and the cyclic group acts freely on the invertible
grading parameter, so the quotient chart is regular.
Since \(S\) is Cohen--Macaulay, Serre's criterion proves that \(S\) is normal.
\end{proof}

We now express \(S\) as the section
ring of a \(\QQ\)-divisor on \(B=\PP^2\) using Demazure's description of normal graded rings
\cite[Introduction]{Watanabe}: since the degrees of the
nonzero pieces of \(S\) generate \(\ZZ\) by Lemma~\ref{lem:leafquotientnormal},
choosing a degree-one rational parameter identifies \(S_m\) with \(H^0(B,\OO_B(\lfloor mQ\rfloor))\)
for a \(\QQ\)-divisor \(Q\) on \(B\).

\begin{lemma}\label{lem:exactdemazure}
Choose a line \(H_0\) on \(B=\PP^2\). Up to changing the
degree-one rational parameter, the algebra \(S\) is
\[
 S=\bigoplus_{m\ge0}
 H^0\!\left(\PP^2,\OO_{\PP^2}(\lfloor mQ\rfloor)\right),
 \qquad
 Q=H_0-\frac12L_2-\frac13L_3-\frac17L_7.
\]
If the three lines are not concurrent, then
\[
 S\simeq\CC[u_6,v_{14},w_{21}].
\]
If they are concurrent, then
\[
 S\simeq
 \CC[u_6,v_{14},w_{21},t_{42}]/(u_6^7+v_{14}^3+w_{21}^2).
\]
\end{lemma}

\begin{proof}
At a line \(L_r\) of order \(r\in\{2,3,7\}\), we use the local
transverse cyclic chart
\[
 xy=z^r,\qquad B=(x=z=0),
\]
where \(Q_1=\OO_G(B)/\OO_G\) is generated by
\(z^{r-1}/x\) and \(Q_r\) is generated by \(1/x\), with
\[
 \left(\frac{z^{r-1}}x\right)^r
       =y^{r-1}\frac1x.
\]
Thus the fractional coefficient of \(Q\) along \(L_r\)
is \(-1/r\) modulo \(\ZZ\), and by Lemma~\ref{lem:leafcurves} \(Q\) has no other
fractional coefficients.
Since \(42Q\sim H_0\) by Lemma~\ref{lem:leafquotientnormal} and \(\Pic(\PP^2)=\ZZ[H_0]\), we obtain
\[
 Q=H_0-\frac12L_2-\frac13L_3-\frac17L_7,
 \qquad \deg Q=\frac1{42}
\]
up to adjusting the degree-one parameter by a rational function.

Choose linear forms
\(l_2,l_3,l_7,h\) defining \(L_2,L_3,L_7,H_0\), and set
\[
 d_m=m-\left\lceil\frac m2\right\rceil
        -\left\lceil\frac m3\right\rceil
        -\left\lceil\frac m7\right\rceil,\qquad
 N_m=\frac{l_2^{\lceil m/2\rceil}
           l_3^{\lceil m/3\rceil}
           l_7^{\lceil m/7\rceil}}{h^m}.
\]
Suppressing the formal degree-\(m\) parameter, multiplication by \(N_m\) gives
\[
 S_m=N_mH^0(\PP^2,\OO_{\PP^2}(d_m)),
\]
which is zero when \(d_m<0\). Putting
\(u=N_6\), \(v=N_{14}\), and \(w=N_{21}\), direct multiplication gives
\begin{equation}\label{eq:demazureproducts}
 u^7=l_7N_{42},\qquad
 v^3=l_3N_{42},\qquad
 w^2=l_2N_{42}.
\end{equation}
Every residue class \(m\bmod42\) is represented by a unique triple
\[
 0\le a<7,\qquad 0\le b<3,\qquad 0\le c<2,\qquad
 m_0=6a+14b+21c\equiv m\pmod{42},
\]
for which evaluating the ceiling functions gives
\[
 d_{m_0}=0,\qquad N_{m_0}=u^av^bw^c.
\]
Writing \(m=m_0+42k\) with \(k\ge 0\), we have \(d_m=k\) and \(N_m=u^av^bw^cN_{42}^k\), so
\begin{equation}\label{eq:residuedecomposition}
 S_m=u^av^bw^cN_{42}^k H^0(\PP^2,\OO_{\PP^2}(k)).
\end{equation}

If \(L_2,L_3,L_7\) are not concurrent, \(l_2,l_3,l_7\) form a basis of \(H^0(\PP^2,\OO_{\PP^2}(1))\), so
\eqref{eq:demazureproducts} and
\eqref{eq:residuedecomposition} show that \(u,v,w\)
generate \(S\); since \(S\) is a three-dimensional domain, \(S\simeq\CC[u_6,v_{14},w_{21}]\).
If \(L_2,L_3,L_7\) are concurrent, \(l_2,l_3,l_7\) span a
two-dimensional subspace with one linear relation having three
nonzero coefficients, which after rescaling \(u,v,w\) becomes \(u^7+v^3+w^2=0\).
Choosing a line \(l_0\) outside their span and setting
\(t=l_0N_{42}\), \eqref{eq:residuedecomposition} shows that \(u,v,w,t\)
generate \(S\), and since \(u^7,v^3,w^2,t\) generate \(S^{(42)}\) subject only to \(u^7+v^3+w^2=0\), we obtain the stated presentation.
\end{proof}

\subsection{Excluding concurrent lines}

\begin{proof}[Completion of the proof of Proposition~\ref{prop:leaf42}]
Suppose first that the three lines are concurrent. Lifting
\(u,v,w,t\) to homogeneous elements of \(R\) and adjoining \(s\) gives a graded surjection
\[
 \CC[s_1,u_6,v_{14},w_{21},t_{42}]\longrightarrow R
\]
by induction on the degree.
Its kernel is a height-one homogeneous prime ideal generated by a single weighted-homogeneous polynomial \(F\) of degree \(42\) with
\[
 F(0,u,v,w,t)=u^7+v^3+w^2
\]
after rescaling. Since the only monomial of degree \(42\) involving \(t\) is \(t\) itself, the coefficient of \(t\) in \(F\) is zero, so \(F=F(s,u,v,w)\) is independent of \(t\).

Consider the point \([0:0:0:0:1]\) of the resulting hypersurface \(G\subset\PP(1,6,14,21,42)\). Its standard
weighted-projective chart is the \(\mu_{42}\)-quotient of the affine threefold
\[
 W=\{F(s,u,v,w)=0\}\subset\mathbb A^4
\]
with weights \((1,6,14,21)\).
Since \(R=\OO(W)[t]\) is normal, \(W\) is normal.
Let us check that \(W\to G\) is a quasi-\'etale cover of this chart.
In \(\mathbb A^4\), only elements of \(\mu_{42}\) of orders \(2,3,7\) have fixed coordinate planes of codimension two, namely \(s=w=0\), \(s=v=0\), and \(s=u=0\) (all other nontrivial fixed loci are contained in these planes).
These three coordinate planes intersect \(W\), respectively, in the plane curves
\[
 u^7+v^3=0,\qquad u^7+w^2=0,\qquad v^3+w^2=0,
\]
each of which has codimension two in \(W\).
Thus the \(\mu_{42}\)-action on \(W\) is free in codimension one, so \(W\) is canonical by \cite[Proposition 5.20]{KM}.
However, taking the weighted blowup of \(\mathbb A^4\) with primitive weights \((1,6,14,21)\) gives, by adjunction, an irreducible exceptional divisor \(E\) over \(W\) with discrepancy
\[
 a(E,W)=1+6+14+21-1-\deg F=-1,
\]
contradicting canonicity. Thus the three lines are not concurrent.

Lemma~\ref{lem:exactdemazure} now gives \(R/(s)\simeq\CC[u_6,v_{14},w_{21}]\).
Lifting these three generators to \(R\) and adjoining \(s\) yields a graded surjection
\[
 \CC[s_1,u_6,v_{14},w_{21}]\longrightarrow R
\]
between four-dimensional domains, which is an isomorphism. Taking \(\Proj R\) completes the proof of Proposition~\ref{prop:leaf42}.
\end{proof}

\subsection{Ordinary adjunction on the fourfold}

We return to \(X\) with \(V=3528\). By
Proposition~\ref{prop:canonicalsharp}, the normalization \(\nu:G\to D\) of a general member \(D\) of the pencil satisfies
\[
 L=\nu^*(H|_D),\quad K_G\sim-L,\quad L^3=42,\quad
 L\qsim42B,\quad L^2B=1,\quad H\qsim84D.
\]

\begin{lemma}\label{lem:normalmember}
The divisor \(D\) is normal. Under the identification
in Proposition~\ref{prop:leaf42}, its ordinary different is
\[
 \operatorname{Diff}_D(0)=\frac12B.
\]
The transverse singularity of \(X\) at the generic point
of \(B\) is \(A_1\), and \(X\) is smooth at the generic
point of every other prime divisor of \(D\).
\end{lemma}

\begin{proof}
Ordinary adjunction on the normalization gives
\[
 K_G+\operatorname{Diff}_{G}(0)
 \qsim(K_X+D)|_G.
\]
Using \(K_G\sim-L\) and \(D|_G\qsim L/84\), we obtain
\[
 L^2\operatorname{Diff}_{G}(0)
 =L^2(D|_G)=\frac{L^3}{84}=\frac12.
\]
Every prime divisor on \(G\) has positive integral
\(L^2\)-degree, so Lemma~\ref{lem:ordinarydifferent}
forces \(\operatorname{Diff}_G(0)=\frac12 E\) for a single prime divisor
\(E\) with \(L^2E=1\), and implies that \(D\) is normal
with the stated transverse singularities along \(D\).
By Proposition~\ref{prop:leaf42}, \(\Cl(G)=\ZZ[B]\) and \(h^0(G,B)=1\), so \(L^2E=L^2B=1\) gives \(E=B\).
\end{proof}

\subsection{Reconstruction of the fourfold section ring}

\begin{lemma}\label{lem:fourfoldquotient}
Let
\[
 R_X=\bigoplus_{m\ge0}H^0(X,\OO_X(mD)),
\]
and let \(s\in(R_X)_1\) define \(D\). Then
\[
 R_X/(s)\simeq\CC[t_1,v_{12},w_{28},z_{42}]
\]
as graded algebras.
\end{lemma}

\begin{proof}
For \(m\ge0\), put
\[
 \mathcal Q_m=\OO_X(mD)\big/s\OO_X((m-1)D).
\]
By Lemma~\ref{lem:divisorialCM} and the depth lemma,
\(\mathcal Q_m\) is a rank-one Cohen--Macaulay sheaf on the normal threefold \(D\), and is therefore reflexive. We can thus compute the sheaves \(\mathcal Q_m\) and their multiplication in codimension one on \(D\).

At the generic point of every prime divisor of \(D\) other than \(B\), \(X\) is smooth and \(D\) is Cartier by Lemma~\ref{lem:normalmember}, so this is ordinary Cartier restriction.
At the generic point of \(B\), Lemma~\ref{lem:normalmember}
gives the local transverse \(A_1\) model
\[
 U=\Spec\CC[x,y,z]/(xy-z^2),\qquad
 D=(x=z=0),\qquad B=(y=0)\text{ on }D,
\]
where in fractional-module notation,
\[
 \OO_U(D)=\OO_U+\OO_U\frac zx,\qquad
 \OO_U(2D)=x^{-1}\OO_U.
\]
Thus \(\mathcal Q_1\) is locally generated by \(z/x\),
\(\mathcal Q_2\) by \(1/x\), and
\begin{equation}\label{eq:paritymultiplication}
 \left(\frac zx\right)^2=y\,\frac1x.
\end{equation}
Since \(2D\) is Cartier at the generic point of every prime divisor of \(D\) and restricts to
\[
 (2D)|_D\qsim\frac1{42}L\sim B
\]
in \(\Cl(D)=\ZZ[B]\), we have
\[
 \mathcal Q_2\simeq\OO_D(B),\qquad
 \mathcal Q_1^{[2]}\simeq\OO_D((2D)|_D-B)\simeq\OO_D,
\]
and hence \(\mathcal Q_1\simeq\OO_D\) because \(\Cl(D)\) is torsion-free. Repeating this local calculation and extending reflexively gives
\begin{equation}\label{eq:fourfoldquotientsheaves}
 \mathcal Q_m\simeq\OO_D(\lfloor m/2\rfloor B)
 \qquad(m\ge0),
\end{equation}
where the product of two odd-degree pieces acquires the defining section of \(B\) by
\eqref{eq:paritymultiplication}, and all other products are ordinary products of divisorial sheaves on \(D\).

For every \(m\ge0\), \((m-1)D-K_X\qsim(m+83)D\) is ample, so the integral \(\QQ\)-Cartier Weil-divisor form of Kawamata--Viehweg vanishing
\cite[Theorem 1.2]{FujinoKV} gives
\(H^1(X,\OO_X((m-1)D))=0\), and hence
\[
 R_X/(s)\simeq\bigoplus_{m\ge0}H^0(D,\mathcal Q_m).
\]
Writing \(R(D,B)=\CC[u_1,v_6,w_{14},z_{21}]\) from Proposition~\ref{prop:leaf42},
where \(u_1\) defines \(B\), \eqref{eq:fourfoldquotientsheaves} and \eqref{eq:paritymultiplication} give
\[
 R_X/(s)\simeq R(D,B)^{[2]}[t_1]/(t_1^2-u_1),\qquad \deg t_1=1,
\]
where \(R(D,B)^{[2]}\) denotes \(R(D,B)\) with all degrees doubled (after rescaling \(u_1\) to absorb any nonzero constant). Eliminating \(u_1\) gives \(R_X/(s)\simeq\CC[t_1,v_{12},w_{28},z_{42}]\).
\end{proof}

\begin{proof}[Proof of the equality assertion in
Theorem~\ref{thm:canonical}]
Assume \(V=3528\). Lifting the four generators of
\(R_X/(s)\) from Lemma~\ref{lem:fourfoldquotient} to \(R_X\) and adjoining \(s\) gives a graded surjection
\[
 \CC[S,T,V_0,W,Z]\longrightarrow R_X,\qquad
 \deg(S,T,V_0,W,Z)=(1,1,12,28,42).
\]
Since both rings have dimension five and \(R_X\) is a domain, this map is an isomorphism, and because \(D\) is ample
\(\QQ\)-Cartier,
\[
 X\simeq\Proj R_X=\PP(1,1,12,28,42),
\]
with \(\Cl(X)=\ZZ[D]\) and \(-K_X\sim84D\).
\end{proof}

\subsection{A gap below the maximal volume}

\begin{corollary}\label{cor:volumegap}
Under the hypotheses of Theorem~\ref{thm:canonical}, if
\((-K_X)^4\ge3130\), then
\(X\simeq\PP(1,1,12,28,42)\). Thus every other member of
this class has anticanonical volume at most \(3129\).
\end{corollary}

\begin{proof}
Put
\[
 t_0=\frac{256}{7161},\qquad
 M_0=\frac{7161^2}{16384},
\]
so that
\[
 \frac{5625}{64t_0(1-6t_0)}=M_0,\qquad
 \frac{5625t_0}{128(1-6t_0)}=2,\qquad
 3129<M_0<3130.
\]
Since \(V\) is an integer, \(V\ge3130\) implies \(V>M_0\). Because \(16875/8<M_0\), the pencil defined by the rank-three HN term of \(\TT_X\) exists, and \(\tau<t_0\) (otherwise \eqref{eq:canonicalmajorant} and \(\eta V\le B(\lambda)\) would give \(V\le M_0\)).
Thus
\[
 \tau\ell=\frac{\eta\tau^2V}{2}
 \le\frac{5625\tau}{128(1-6\tau)}<2,
\]
and Lemma~\ref{lem:smallmass} applies.

A nonreduced fiber would give \(V\le2646\) by Proposition~\ref{prop:canonicalsharp}, so \(\eta=1\), \(V=2\ell^2/b\) with \(b\ge1\), and
\(\ell\in\{1,\ldots,22,24,30,42\}\).
If \(\ell\ne42\), then \(V\le2\cdot30^2=1800\),
a contradiction. Thus
\[
 \ell=42,\qquad b=\frac{3528}{V},\qquad 1\le b<\frac43.
\]
Ordinary adjunction on the normalization \(G\) of \(D\), with
\(D\qsim(\tau/2)H\), gives
\[
 L^2\operatorname{Diff}_{G}(0)
 =L^2(D|_G)=\frac{\tau\ell}{2}
 =\frac b2\in\left[\frac12,\frac23\right).
\]
Since every prime divisor on \(G\) has positive integral \(L^2\)-degree, Lemma~\ref{lem:ordinarydifferent} forces \(L^2\operatorname{Diff}_G(0)=1/2\), so \(b=1\), \(V=3528\), and Theorem~\ref{thm:canonical} gives \(X\simeq\PP(1,1,12,28,42)\).
\end{proof}

\section{Bounds from the non-Cartier locus}\label{sec:finite}

We now study the case where the dimension of \(\NCart(A)\) is restricted. By Lemma~\ref{lem:divisibility}, the dimension of \(\NCart(A)\) controls the divisibility of \(V\):
\[
\begin{array}{c|cccc}
\NCart(A)&\varnothing&\dim\le0&\dim\le1&\dim\le2\\ \hline
\text{divisibility}&q^4\mid V&q^3\mid V&q^2\mid V&q\mid V.
\end{array}
\]
Throughout the next three subsections, we assume that \(A\) is Cartier outside finitely many points (\(\dim\NCart(A)\le 0\)) and prove Theorem~\ref{thm:finite}.

\subsection{The numerical reduction for a finite locus}
For \(q\ge12\), one has \(\lambda_4\ge1/q\): if the last
HN block has rank one, this follows from its positive integral
determinant degree, while if its rank is two, three, or four,
Proposition~\ref{prop:canonicalslopes} gives \(1/12,1/6,\) or \(1/4\). Thus
\((1/2,1/3,1/6-1/q,1/q)\) majorizes \(\lambda\),
and Lemma~\ref{lem:jet-equality} and Theorem~\ref{thm:jets} give
\begin{equation}\label{eq:continuousq}
 V\le\frac{5625q^2}{64(q-6)}.
\end{equation}
Combining \eqref{eq:continuousq} with \(V\ge q^3\)
gives \(q(q-6)\le5625/64\), so \(q\le12\).
Enumerating the positive block compositions
\eqref{eq:blocks} satisfying \eqref{eq:canonicalslopes} and rounding
\(B(\lambda)\) down to a multiple of \(q^3\) gives the preliminary numerical upper bounds in
Table~\ref{tab:finitecaps}.
\begin{table}[htbp]
\centering
\caption{Numerical upper bounds before applying the finite
section estimates and analysing equality.}\label{tab:finitecaps}
\begin{tabular}{@{}cr@{\qquad\qquad}cr@{}}
\toprule
\(q\)&Upper bound&\(q\)&Upper bound\\
\midrule
1&625&7&686\\
2&1048&8&1024\\
3&783&9&729\\
4&1216&10&1000\\
5&1500&11&1331\\
6&1944&12&1728\\
\bottomrule
\end{tabular}
\end{table}

Suppose \(V>1024\), so \(q\in\{2,4,5,6,11,12\}\) by Table~\ref{tab:finitecaps}.
For every admissible vector \(\lambda\) in these cases, \(B(\lambda)<2V\), so
Lemma~\ref{lem:cover} gives \(H\sim qA\) and Corollary~\ref{lem:WeilRR} applies.

When \(q>2\) and \(N=h^0(A)\ge3\), \eqref{eq:Weilsmall} and \eqref{eq:quadraticgrowth} give
\begin{equation}\label{eq:finiteNgrowth}
 h^0(2A)=2N+k\quad(k=V/q^3),\qquad
 N\le\frac{k+6}{2}.
\end{equation}
The vectors \(\lambda\) at \(q\in\{5,6,11,12\}\) with
\(q^3\lfloor B(\lambda)/q^3\rfloor>1024\) are:
\[
\begin{array}{c|c|c|c}
q&\lambda&\text{Chern lower bound for }N/k&
       \text{Finite-jet upper bound for }N\\ \hline
5&(2/5,2/5,1/10,1/10)&41/48&10\\
6&(5/12,5/12,1/12,1/12)&9/8&10\\
6&(1/2,1/3,1/12,1/12)&53/48&10\\
11&(9/22,9/22,1/11,1/11)&55/16&4\\
11&(5/11,4/11,1/11,1/11)&41/12&4\\
12&(5/12,5/12,1/12,1/12)&4&4\\
12&(1/2,1/3,1/12,1/12)&47/12&4,
\end{array}
\]
where the third column comes from \(N=k/6+C/(12q)\) and \eqref{eq:chern}, and the fourth column records
\(N(\lambda,q):=\max \{n:\J_\lambda(n)\leq n/q\}\).
For \(q=5\), \(k\ge9\), and \(41k/48\le(k+6)/2\) gives \(k\le144/17<9\).
For \(q=6\), \(k\ge5\), and \(53k/48\le(k+6)/2\) gives \(k\le144/29<5\).
For \(q=11,12\), Table~\ref{tab:finitecaps} forces \(k=1\), and the above table forces \(N=4\), giving \(h^0(2A)=9<10=4N-6\), a contradiction.

For \(q=2\), the only vector \(\lambda\) with \(B(\lambda)>1024\) is
\((1/2,1/6,1/6,1/6)\). Positive degree compositions of \(2\)
for the rank-five extension \(\EE_X\) are
\((5;2)\), \((1,4;1,1)\), and \((2,3;1,1)\), each having \(\sum_i\nu_i^2\le 5/16\).
Lemma~\ref{lem:chern} for \(\EE_X\) gives \(C\ge11V/32\), so
\[
                             P_1\ge1+\frac{25}{128}V.
\]
For integer weights \((3,1,1,1)\) and allowed average weight \(6\), Theorem~\ref{thm:jets} gives \(P_1\le202\), so \(V\le\lfloor201\cdot128/25\rfloor=1029\), and \(8\mid V\) gives \(V\le1024\).

For \(q=4\), the vectors \(\lambda\) with \(B(\lambda)>1024\) are
\((3/8,3/8,1/8,1/8)\) and \((1/2,1/4,1/8,1/8)\).
For the first, \(N\ge5k/8\) and \(N\le10\), whereas $V=64k>1024$ gives \(k\ge17\) and $N\geq 85/8>10$. For the second,
\(29k/48\le N\le11\) and \(17\le k\le19\), which excludes $k=19$ and leaves \(k=17,18\):
\begin{equation}\label{eq:finite4}
\begin{array}{c|r|r|r|c}
V&N&C&P_1&\lambda\\ \hline
1088&11&392&215&(1/2,1/4,1/8,1/8)\\
1152&11&384&225&(1/2,1/4,1/8,1/8).
\end{array}
\end{equation}

\subsection{Excluding \(V=1152\) using the canonical extension}
For \(q=4\), a rank-one first HN block of \(\EE_X\) has determinant
degree at most two by \eqref{eq:anchorpolygon} and Proposition~\ref{prop:canonicalslopes}, leaving the compositions
\[
\begin{split}
(r_j;c_j)={}&(5;4),\ (1,4;1,3),\ (1,4;2,2),\
 (2,3;2,2),\\
 & (2,3;3,1),\ (3,2;3,1),\ (1,1,3;2,1,1).
\end{split}
\]
The maximum of \(\sum_i\nu_i^2\) is \(1/3\), attained uniquely by
\((1,1,3;2,1,1)\), so Lemma~\ref{lem:chern} gives \(C\ge\frac{1-1/3}{2} V=V/3\). For \((V,C)=(1152,384)\), equality holds, so \(\EE_X\) has type \((1,1,3;2,1,1)\) and its final rank-three reflexive HN quotient has determinant \(A\) and zero orbifold discriminant. Since \(B(1/2,1/4,1/8,1/8)=1250<2V\), Lemma~\ref{lem:cover} and
Lemma~\ref{lem:flat} make \(A\) divisible by three in \(\Cl(X)\), contradicting the primitivity of $A$.

\subsection{Excluding \(V=1088\) via the image of \(|A|\)}
We first record an intersection estimate for rational maps whose images have dimension \(n-1\).

\begin{lemma}\label{lem:degreebridge}
Suppose \(D\equiv tH\) is an effective \(\QQ\)-Cartier Weil
divisor on a normal projective \(n\)-fold, \(H\) is ample
Cartier, and the image \(Y\) of \(|D|\) has dimension \(n-1\).
If the general fiber cycle has \(H\)-degree at least \(\ell\),
then
\[
                              \ell\deg Y\le t^{n-1}H^n.
\]
\end{lemma}

\begin{proof}
Resolve the rational map and the base ideal of \(|D|\) via
\(\pi:Z\to X\), with moving divisor \(M\) and morphism \(g:Z\to Y\).
Pulling back effective Cartier multiples of \(D\) shows that \(\pi^*D\sim_{\QQ}M+E\) with \(E\ge0\).
Since \(\pi^*D\equiv t\pi^*H\) and \(M\) are nef, intersecting \(Y\) with \(n-1\) general hyperplanes and applying the projection formula gives
\[
 t^{n-1}H^n
 \ge \pi^*H\,M^{n-1}
 =(\deg Y)\,\deg_H(\text{general fiber cycle})
 \ge\ell\deg Y.
\]
\end{proof}

For \(V=1088\), scaling \(\lambda=(1/2,1/4,1/8,1/8)\) by \(8\) gives integer weights \((4,2,1,1)\) and allowed average weight \(2\)
for \(A\). The \(11\) monomials of smallest weight consist of \(1\), two monomials of weight \(1\),
four of weight \(2\), and four of weight \(3\), with total weight
\(2+8+12=22=11\cdot 2\). Thus equality holds in \eqref{eq:finitejet}. Among these \(11\) monomials, exactly three have ordinary total degree one (the two variables of weight one and the variable of weight two; the variable of weight four cannot occur). Lemma~\ref{lem:jet-equality} therefore gives
\(\dim Y=3\) for \(Y=\im(|A|)\subset\PP^{10}\).

Let \(\overline{Q}_1\) be the saturation of the image of \(Q_1\) in \(\EE=\Omega_X^{[1]}\,\widehat\otimes\,\OO_X(A)\). Equality in \eqref{eq:finitejet} gives \(c_1(Q_1)H^3/V=1/4\). The normalized HN slopes of \(\EE\) are \(1/8,1/8,0,-1/4\). Saturation cannot decrease the degree, and the HN polygon gives the upper bound \(1/4\). Thus \(\overline{Q}_1\) attains this bound. The strict gap between the third and fourth slopes implies that \(\overline{Q}_1\) is the rank-three HN term. After twisting by \(-A\), this term is the annihilator of the maximal destabilizing subsheaf \(\FF\subset\TT_X\). Since \(Q_1\) and \(\overline{Q}_1\) agree at a general point, their annihilators agree there as well.

At a general smooth point $x\in X$, choose $s_0\in H^0(X,A)$ with $s_0(x)\neq 0$. In the resulting local trivialization of $A$, we have
\[
Q_1(-A)_x=\operatorname{Span} \{d(s/s_0)_x:s\in H^0(X,A)\},
\]
and hence
\[
\ker(d\phi_{|A|,x})=\operatorname{Ann}(Q_1(-A)_x)=\operatorname{Ann}(\overline{Q}_1(-A)_x)=\FF_x.
\]
Thus the general fibers of $\phi_{|A|}$ are tangent to the rank-one foliation $\FF$, which has $-K_{\FF}\equiv \frac12 H$.

A normalized general leaf of $\FF$ is
\(\PP^1\), and the proof of Proposition~\ref{prop:canonicalslopes}
gives \(\deg L\ge2\) and \(\frac12\deg L\le1\), so \(\deg L=2\).
Every general fiber cycle of $\phi_{|A|}$ therefore has \(H\)-degree
at least two, and Lemma~\ref{lem:degreebridge} with \(t=1/4\) and \(\ell=2\) gives
\[
                         2\deg Y\le V/4^3=17.
\]
Since a nondegenerate threefold in \(\PP^{10}\) has degree at least
$10-3+1=8$, \(Y\) has degree \(8\) and is a variety
of minimal degree of codimension \(7\) in \(\PP^{10}\). By the classification of varieties of minimal degree
\cite{EH} (see also \cite[Theorem 0.1]{EGHP}),
\(Y\) is a rational normal scroll (possibly a cone), since
a quadric has codimension one and a cone over the Veronese
surface has codimension three.
The scroll equations include a nonzero \(2\times2\) minor of
a matrix of linear forms, which has rank at
most four; pulling it back to \(X\) gives a nonzero rank-at-most-four relation in
\(\Sym^2H^0(A)\), contradicting Lemma~\ref{lem:primitive}(2).
This excludes \(V=1088\) in \eqref{eq:finite4} and proves \(V\le1024\).

The weighted projective space \(\PP(1,1,1,1,4)\) is well-formed, \(\QQ\)-factorial, and
Gorenstein of Picard number one, with a single isolated canonical quotient
singularity \(\frac14(1,1,1,1)\) (of nontrivial ages
\(1,2,3\)), anticanonical volume
\(8^4/4=1024\), and primitive polarization \(A=\OO(1)\) Cartier outside the singular point. This completes the proof of
Theorem~\ref{thm:finite}.

\subsection{An example with a one-dimensional non-Cartier locus}

\begin{example}\label{ex:curve}
Let \(X=\PP(1,1,1,3,6)\) and \(A=\OO_X(1)\).
This weighted projective space is well-formed, \(\QQ\)-factorial,
of Picard number one, and Gorenstein canonical since every weight divides \(12\). The non-Cartier locus of \(A\) is the weighted line \(\PP(3,6)\), and
\[
 q=12,\qquad V=\frac{12^4}{3\cdot6}=1152.
\]
Thus \(q^2\mid V\) while \(q^3\nmid V\), showing that the bound \(1024\) of Theorem~\ref{thm:finite}
does not extend unchanged when \(\dim\NCart(A)=1\).
\end{example}

\section{Further discussion}\label{sec:discussion}

Theorems~\ref{thm:terminal} and \ref{thm:canonical} determine
the optimal anticanonical volumes in the terminal and canonical
\(\QQ\)-factorial Gorenstein classes with Picard number one.
It is natural to ask whether the canonical bound \(3528\)
continues to hold at arbitrary Picard number, as supported by the toric
theorem of Balletti--Kasprzyk--Nill \cite{BKN}.

Our proof indicates the role of each hypothesis.
The Cartier property of \(H=-K_X\) gives the integral discrepancy
\(R_G\) and the positive integral intersection numbers \(L^2C\) in
\eqref{eq:massidentity}. Picard number one makes both the reflexive normal
sheaf \(N_{\FF}=(\TT_X/\FF)^{**}\) and the general pencil member \(D\) \(\QQ\)-proportional to \(H\). This yields
\(\ell=\eta\tau V/2\) and \(L\qsim\ell B\), connecting
the ambient volume \(V\) to the threefold Weil index \(\ell=q_{\QQ}(G)\). At higher Picard
number, slope inequalities with respect to \(H^3\) alone do not give these divisor
relations, and one would also need a relative estimate controlling the
restriction of \(N_{\FF}\) to the general leaf.

Proposition~\ref{prop:normaljets} gives an asymptotic jet bound on any normal projective variety with positive normalized tangent slopes. The section-ring computation in Section~\ref{sec:canonical-equality} uses the \(\QQ\)-Cartier property of \(D\) and the specified adjunction data; extending the volume theorem beyond the present setting would also require the preceding foliation and divisor-class arguments. Note also that passing to a small \(\QQ\)-factorialization can
increase the Picard number and gives a weak Fano model,
so it does not directly reduce the non-\(\QQ\)-factorial case to
Theorem~\ref{thm:canonical}.

For canonical Gorenstein fourfolds of higher Cartier Fano
index, adjunction to a Cartier fundamental divisor provides
additional reductions. The Cartier index-one case has no such
smaller Cartier polarization; while Liu's positivity and
nonvanishing results \cite{Liu} still apply in
this setting, one needs to relate the section
spaces and the reflexive normal sheaf \(N_{\FF}\) of a pencil to the geometry
of its leaf. We leave these extensions for future work.

\clearpage

\end{document}